\documentclass[12pt]{amsart}

\usepackage{amsmath,amssymb,amscd}
\usepackage{epsfig}
\usepackage{latexsym, url} 
\usepackage[hidelinks]{hyperref}
\usepackage[scale=0.7]{geometry}

\theoremstyle{plain}
\newtheorem{theorem}{\bf Theorem}[section]
\newtheorem*{theorem*}{\bf Theorem}
\newtheorem{lemma}[theorem]{\bf Lemma}
\newtheorem{definition}[theorem]{\bf Definition}
\newtheorem{proposition}[theorem]{\bf Proposition}
\newtheorem*{proposition*}{\bf Proposition}
\theoremstyle{remark}
\newtheorem{example}[theorem]{\bf Example}
\newtheorem{corollary}[theorem]{\bf Corollary}

\theoremstyle{remark}
\newtheorem{remark}[theorem]{\bf Remark}

\def\N{{\mathbb N}}
\def\Q{{\mathbb Q}}
\def\Z{{\mathbb Z}}
\def\R{{\mathbb R}}
\def\C{{\mathbb C}}

\def\T{{\mathbb T}}

\def\e{{\varepsilon}}
\def\trans{{\rm Trans}}
\def \eqdef{:=}
\def \cal{\mathcal}

\def\d{{{\rm d}}}
\def\ds{\displaystyle}
\def\p{\phi}

\title{Boundaries of the Arnol'd tongues and the standard family}

\author{Kuntal Banerjee}\thanks{Work mainly supported by EU Research Training Network on Conformal Structures and Dynamics (CODY) and partially by CNRS and  the grant ANR-08-JCJC-0002.}
\email{kuntalb@gmail.com}
\address{Institut de Math\'ematiques de Toulouse,
Universit\'e Paul Sabatier, 118 route de Narbonne, 31062 Toulouse Cedex 9, France.}
\curraddr{School of Mathematics, Harish-Chandra Research Institute, Chhatnag Road, Jhusi, Allahabad 211019, India.}

\begin{document}

\maketitle

\begin{abstract}
For a family $(F_{t,a} : x \mapsto x + t + a\phi(x))$
of increasing homeomorphisms of $\mathbb R$ with $\phi$ being Lipschitz continuous of period 1, 
there is a parameter space consisting of the values $(t,a)$ such that the map $F_{t,a}$ is strictly increasing
and it induces an orientation preserving circle homeomorphism. For each $\theta \in \mathbb R$
there is an \textsf{Arnol'd tongue} $\mathcal T_\theta$ of \textsf{translation number} $\theta$ in the parameter space.
Given a rational $p/q$, it is shown that the boundary $\partial \mathcal T_{p/q}$ is a union of two Lipschitz curves which intersect at $a=0$ and there can be a non zero angle between them. In this direction we compute the first order asymptotic 
expansion of the boundaries of the rational and irrational tongues in the parameter space around $a=0$.  

For the standard family $(S_{t,a} : x \mapsto x + t + a \sin(2\pi x))$, the boundary
curves of $\mathcal T_{p/q}$ have the same tangency at $a=0$ for $q\ge 2$ 
and it is known that $q$ is their \textsf{order of contact}. 
Using the techniques of \textsf{guided} and \textsf{admissible family}, we give a new proof of this. In particular
we relate this to the \textsf{parabolic multiplicity} of the map $s_{p/q} : z \mapsto e^{i2\pi p/q}ze^{\pi z}$
at $0$.
\end{abstract}

\bigskip

Mathematics Subject Classes : 37E10; 26A18; 30D05
\bigskip

\tableofcontents

\section{Introduction}

In this article we study certain structures in the parameter space of families of circle homeomorphisms. 
Poincar\'e was the first person to study the dynamics of circle homeomorphisms while he was looking at the solutions of differential equations on torus in his 1885 m\'emoire \cite{poincare}. Formally, if we start with an increasing homeomorphism $F:\R \to \R$ of the real line such that $F(x+1)=F(x)+1$ for any $x\in \R$, then $F$  will induce an orientation preserving circle homeomorphism $f: \T \to \T$ given by $[x] \mapsto [F(x)]$. Note that we are denoting the circle as the additive group $\T=\R/\Z$. A natural question is how much every point on the real line is translated on average under $F$ or how much every point on the circle is rotated on average under the action of $f$. If we want to look at the average displacement after the $n$-th iterate of $F$ at point $x\in \R$, then we should look at the quantity $\ds \frac{F^{\circ n}(x)-x}{n}$. Poincar\'e showed that this quantity has a limit as $n\to \infty$ and the limit does not depend on the choice of the point $x$. We will call this limit as the \textsf{translation number} of $F$ and denote this as $\trans(F)$. The translation number modulo 1 will be defined as the \textsf{rotation number} $\rho(f)$. This definition of translation and rotation numbers agrees with the fact that the translation number of a translation $T_\theta : x\mapsto x+\theta$ is $\theta$ and the rotation number of a rotation $R_\theta : [x]\mapsto [x+\theta]$ is equal to $\theta$ mod 1.   

For a fixed $t$, the translation $x\mapsto x+t$ on the real line induces a circle homeomorphism which is equivalent to a rotation. One can perturb this translation with a non-linear factor and also can consider the map
\[F_{t,a}: x \mapsto x+t+a\phi(x)\] 
where $\phi$ is a Lipschitz continuous function of period $1$. 
%Let us assume that 
%\[a_{\min}= {\frac {-1}{{\displaystyle\max_{x,y \in [0,1]}} {\frac {\p(x)-\p(y)} {x-y}}}}\quad\text{and}\quad
%a_{\max}= {\frac {-1}{{\displaystyle\min_{x,y \in [0,1]}} {\frac {\p(x)-\p(y)} {x-y}}}}.\]
There are two constants $a_{\min}$ and $a_{\max}$ which depend on $\phi$ such that if $t\in \R $ and $a_{\min}<a<a_{\max}$ then $F_{t,a}$ is monotone increasing homeomorphism and thus $F_{t,a}$ induces a circle homeomorphism. This is the \textsf{parameter space} of this family, denoted as
\[\cal P \eqdef \{(t,a)~|~t \in \R, a_{\min}<a<a_{\max}\}.\] 
Note that for $(t,a)\in \cal P$, the map $F_{t,a}$ induces an orientation preserving circle homeomorphism $f_{t,a}$. A well known example \cite{Arnold} is the \textsf{standard family} or the \textsf{Arnol'd family} of the form
\[S_{t,a} : x\mapsto x+t+a \sin (2\pi x).\] The parameter space of the standard family is $\cal P_S=\{(t,a)~|~t\in \R, -1/2\pi <a < 1/2\pi \}.$  Arnol'd studied this using the translation number and the rotation
number of the map corresponding to each point in the parameter
space. For $\theta \in \R$ we define the \textsf{Arnol'd tongue}
\[\mathcal T_{\theta} \eqdef \{(t,a)\in \cal P~|~\trans(F_{t,a})=\theta\}.\] 
The collection $\{\mathcal T_\theta\}_{\theta \in \R}$ gives a partition of the parameter space $\cal P$. In this article we shall study the boundaries of these Arnol'd tongues and this is inspired after the work of Arnol'd, Herman, Hall, Boyland and others (see \cite{Arnold}, \cite{Herbig}, \cite{Hall}, \cite{boyland}). \\

For a rational translation number $p/q$, when $p$ and $q$ are coprime and for a fixed $a\in (a_{\min},a_{\max})$, the set of
values of $t$ such that $(t,a) \in \mathcal T_{p/q}$ is an interval
$\left[\gamma_{p/q}^l(a),\gamma_{p/q}^r(a)\right]$. And for an irrational $\alpha$ and a fixed $a \in (a_{\min},a_{\max})$ the set of $t$ such that $(t,a) \in \cal T_{\alpha}$ is singleton $\{\gamma_{\alpha}(a)\}$. These define the \textsf{boundaries of the tongues} and three functions $\gamma_{p/q}^{l} : (a_{\min},a_{\max}) \to \R$, $\gamma_{p/q}^{r} : (a_{\min},a_{\max}) \to \R$ and $\gamma_{\alpha} : (a_{\min},a_{\max}) \to \R$; which one has to study to understand the boundaries of the tongues. The existence of the interval $\left[\gamma_{p/q}^l(a),\gamma_{p/q}^r(a)\right]$ and the singleton set $\{\gamma_{\alpha}(a)\}$ are guaranteed by Lemma \ref{lemma_transincrease} and Proposition \ref{prop_transincrease}.

Note that when $a=0$, the map $F_{t,a}=F_{t,0}:x\mapsto x+t$ is a translation and any rational tongue $\cal T_{p/q}$ grows (see Figure \ref{fig_stdfamily}) from the level $a=0$ in the parameter plane. A rational tongue $\cal T_{p/q}$ usually looks like union of two horn like regions which are meeting each other at a point at the level $a=0$, whereas an irrational tongue in the parameter space is a curve. Hence to understand how these rational tongues are growing from the level $a=0$, one has to understand well the functions $\gamma_{p/q}^{l}, \gamma_{p/q}^r$; which are the boundaries of $\cal T_{p/q}$. It is interesting to study whether $\gamma_{p/q}^l$ intersects $\gamma_{p/q}^r$ at a non zero angle or they have a common tangency at the level $a=0$. In case of common tangency, for $q\ge2$, we can study the \textsf{order of contact} of the boundaries of $\cal T_{p/q}$, which is given by the order of terms upto which the asymptotic expansions of the boundary curves at $a=0$ are identical. This is a difficult question in general, we work on it for the standard family. 

%we say that $k$ is the \textsf{order of contact of the boundaries of} $\cal T_{p/q}$ if
%\[|\gamma_{p/q}^r(a)-\gamma_{p/q}^l(a)|\underset{a\to 0}=\cal O(a^k)\quad\text{but}\quad |\gamma_{p/q}^r(a)-\gamma_{p/q}^l(a)|\underset{a\to 0}\ne\cal O(a^{k+1}).\]

In this article we study few results on the boundaries of the Arnol'd tongues. We derive the first order asymtotic expansion of the boundary curves $\gamma_{p/q}^l$, $\gamma_{p/q}^r$ and $\gamma_{\alpha}$ near $a=0$. This gives the angle of opening for the boundaries of the rational tongues at $a=0$ and this Theorem \ref{prop_angle} is discussed in section \ref{sec_angle}. In the case of standard family the boundaries of the rational tongues $\cal T_{p/q}$ have the same tangency and the order of contact of these boundaries is  obtained in section \ref{sec_ooc} using the technique of guided and admissible family. This is completely a new approach to prove the order of contact and this has further dynamical insight. The order of contact of the boundaries of the rational tongues in the standard family is connected with the parabolic multiplicity (see section 3 for definition) of the guiding family for the first time (see Theorem \ref{thm_orderofcontact1st}). In this way we obtain a result on the characterization of the admissible and guided family of analytic circle diffeomorphisms (see Theorem \ref{theo_characteranalytic}). We also prove that the boundaries of the rational tongue $\cal T_{p/q}$ are analytic curves in the parameter space for the standard family (Theorem \ref{theo_bdanalytic}). We can also derive the order of contact of the boundaries of the  rational tongues in the Blaschke fraction family using this technique of admissible and guided family.  

\begin{figure}[htp]\label{fig_stdfamily}
\centering
\includegraphics[width=5.5in,height=2.5in]{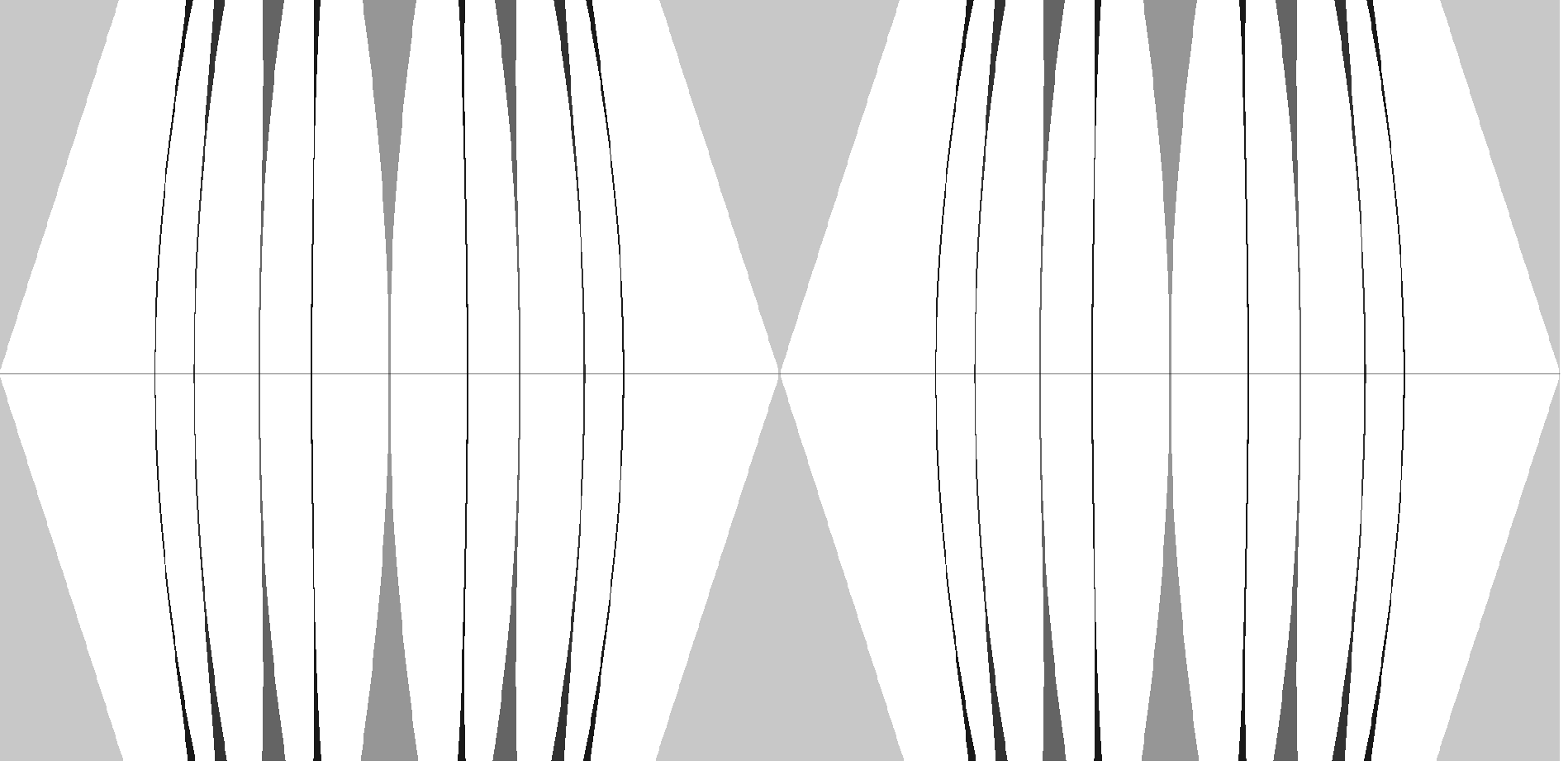}
\caption{Some rational Arnol'd tongues in the standard family.}
\end{figure}

\section{Preliminaries}

Before we proceed further, we recall some basic facts about translation and rotation
numbers. By definition the rotation number is real, so it is either rational or irrational.
The two cases and the corresponding dynamics are discussed in the following results. We
deal with the case of rational rotation number first. Whenever we write $p/q$ for a rational
number, we implicitly assume that $p$ and $q$ are coprime.

\begin{proposition}[Poincar\'e \cite{poincare}]\label{prop_transratpq}
If $\rho(f)\in \Q/\Z$ then $f:\T\to \T$ has a periodic point. More
precisely, if $\trans(F) =p/q\in \Q$ then there is a point $a\in \R$
such that $F^{\circ q}(a)=a+p$.
\end{proposition}

Note that $G\eqdef F^{\circ q}-{\rm Id}-p$ vanishes on the whole
$F$-orbit of $a$, in particular on the set $\{a, F(a), \ldots,
F^{\circ (q-1)}(a)\}$ with $q$ points whose image in $\T$ is a cycle of $f$. We
say that such a cycle has rotation number $p/q$. The
derivative of $G$ is constant along the orbit of $a$ under iteration
of $F$. As $G$ is analytic, either it has a double zero, or it
vanishes at least once with positive derivative and once with
negative derivative. This shows that counting multiplicities, $f$
has at least $2$ cycles with rotation number $p/q$.

Next we address the case of irrational rotation number.

\begin{proposition}[Poincar\'e \cite{poincare}]\label{prop_semiconjugacy}
If $\rho(f) = \alpha \in \R \setminus\Q$, then $f:\T \to \T$ is
semi-conjugate to the rotation $\T \ni [x]\mapsto [x+\alpha] \in
\T$.
\end{proposition}

In fact, the semiconjugacy may be obtained as follows. The sequence
of maps $\Phi_N$ defined as
\[\Phi_N(x)\eqdef \frac{1}{N}\sum_{k=0}^{N-1} \left(F^{\circ k}(x)-F^{\circ k}(0)\right)\]
converges, as $N\to +\infty$, to a non-decreasing continuous
surjective map $\Phi_{F}:\R\to \R$, which satisfies
\[\Phi_{F}(x+1) = \Phi_{F}(x)+1\quad\text{and}\quad \Phi_{F}\circ F(x) =
T_\alpha \circ \Phi_{F}(x)\] where $T_\alpha: x\mapsto x+\alpha$ is the translation by $\alpha$.

The statement of semiconjugacy does not appear in Poincar\'e’s m\'emoire \cite{poincare} but it is
equivalent to the following. In \cite{poincare}, he proved that if $\rho(f)=\alpha$ is irrational, the order
of points in the orbit of $f$ on the circle is the same as the order of points in the orbit
of rotation by $\alpha$ on $\T$.
The following result of Denjoy implies that when $F$ is an analytic
diffeomorphism, then the semiconjugacy is in fact an actual
conjugacy. In other words $\Phi_{F}:\R\to \R$ is an increasing
homeomorphism.

\begin{proposition}[Denjoy \cite{Denjoy}]
If $\rho(f)=\alpha \in \R\setminus\Q$ and if $f$ is a ${\cal C}^2$ diffeomorphism,
then $f:\T\to \T$ is conjugate to the rotation by angle $\alpha$.
\end{proposition}

The following two are basic results which can be found in any standard text in Dynamical Systems (\cite{katok-hasselblatt}, see \cite{thesis} for \ref{prop_transincrease}). Let $\mathcal{D}^+(\mathbb R)$ be the set of increasing homeomorphisms of the real line which induce orientation presetving circle homeomorphisms. For $F, G \in \mathcal{D}^+(\mathbb R)$, we say $F>G$ if $F(x)>G(x)$ for all $x\in \R$.  

\begin{lemma}\label{lemma_transincrease}
Assume that $F,G \in \mathcal{D}^+(\mathbb R)$. If $F>G$, then  $\trans(F) \ge \trans(G)$, and this inequality is strict if one of $\trans(F)$ and $\trans(G)$ is irrational.
\end{lemma}

\begin{proposition} \label{prop_transincrease}
Assume that $F\in \cal D^+(\R)$ and $(F_t)_{t\in {\mathbb R}}$ be the one-parameter family of homeomorphisms of the real line defined as
\[F_t(x):=F(x)+t.\]
Then $t\mapsto \trans(F_t)$ is continuous, non-decreasing and 
\[\trans(F_{t+1})=\trans(F_t)+1.\]
For all $\alpha \in {\mathbb R}, \trans^{-1}(\alpha)=I_{\alpha}$ is a closed interval. If     
$\alpha \in {\mathbb R \setminus \mathbb Q}, I_{\alpha}$ is a point. And for a rational 
$p/q$, $I_{p/q}$ is reduced to a point $t_0 \iff F_{t_0}^{\circ q}(x)-x \equiv p$.
\end{proposition}

\section{Boundary points of $\cal T_{p/q}$ and parabolic fixed points}

In this section our aim is to study the boundaries of the rational Arnol'd tongues in a two parameter real analytic family $(F_{t,a}={\rm Id} + t + a\phi)_{(t,a)\in \cal P}$ where $\phi$ is a real analytic function.  

To begin, let us fix one parameter $a=a_0\in (a_{\min},a_{\max})$ in the two parameter family and for simplicity, let us write $F_t$ for $F_{t,a_0}$. Then, the family$(F_t)_{t\in \R}$ is a one parameter increasing family in the sense that $F_{t}>F_{t'}$ if $t>t'$. 

Let us fix a rational number $p/q$. As mentioned in the previous section, the function $\cal H:t\mapsto \trans(F_t)$ is non decreasing and the set $\cal H^{-1}(p/q)$ is a closed interval $[t^l,t^r]$ by Proposition \ref{prop_transincrease}. We have the following characterization of $t^l$ and $t^r$. 

\begin{lemma}\label{lemma_parabolicNbdry}
We have the following equivalences: 
\begin{itemize}
\item $t=t^r$ if and only if $F_t^{\circ q}(x)\geq x+p$ for all $x\in \R$ and $F_t^{circ q}(x_0)= x_0+p$ for some point $x_0\in \R$.

\item $t=t^l$ if and only if $F_t^{\circ q}(x)\leq x+p$ for all $x\in \R$ and $F_t^{circ q}(x_0)= x_0+p$ for some point $x_0\in \R$.

\item $t\in (t^l,t^r)$ if and only if $F_t^{\circ q}-{\rm Id}-p$ takes both positive and negative values. 
\end{itemize}
\end{lemma}

\begin{proof}
Set $G_t\eqdef F_t^{\circ q}-{\rm Id}-p$.
According to Proposition \ref{prop_transratpq}, $t\in [t^l,t^r]$ if and only if $G_t$ vanishes. 
So, $G_{t^r}$ vanishes a some point $x_0\in \R$. If $t>t^r$, then $\trans(F_t)>p/q$, thus $G_t$ does not vanish. Since $G_t(x_0)>G_{t^l}(x_0)=0$, we have $G_t>0$ for $t>t^r$. Passing to the limit as $t\to t^r$ shows that $G_{t^r}\geq 0$. 
Conversely, if $G_t$ vanishes at some point $x_0\in \R$, then $\trans(G_t)=p/q$ and if in addition $G_t\geq 0$, then $G_{t'}>0$ for $t'>t$ and so, $\trans(F_{t'})>p/q$ for $t'>t$. This shows that $t=t^r$. 

The characterization of $t^l$ follows similarly. 

Finally, if  $G_t$ takes both positive and negative values, then $\trans(F_t)=p/q$ and we cannot be in one of the previous cases, so $t\in (t^l,t^r)$. Conversely, if $t\in (t^l,t^r)$, then $G_t$ vanishes and according to the previous cases, the sign of $G_t$ cannot be constant. Thus, $G_t$ takes both positive and negative values. 
\end{proof}

In particular, we see that when $t=t^l$ or $t=t^r$, any point where $F_t^{\circ q}-{\rm Id}-p$ vanishes is a local extremum of the function. So, if $F_t$ is of class $C^1$ and $t=t^l$ or $t=t^r$, then there is a point $x_0$ such that 
\[F_{t}^{\circ q}(x_0)=x_0+p \quad \text{and}\quad (F_{t}^{\circ q})'(x_0)=1.\]
The induced map $f_t:\T\to \T$ has a periodic cycle whose rotation number is $p/q$ and whose multiplier is $1$. 

\begin{definition}
A fixed point $z_0$ of an analytic map $f$ is said to be \textsf{parabolic} if the multiplier $f'(z_0)$ is a root of unity, i.e. $f'(z_0)=e^{i2\pi p/q}$ for some integers $p$ and $q$, and if $f^{\circ q}$ is not the equal to the identity near $z_0$. It is a \textsf{multiple} fixed point if the multiplier is $1$. 
\end{definition}

It is well known (see \cite{milnor} Section 10 for example) that when  $z_0$ is a parabolic fixed point of $f$ with multiplier $e^{i2\pi p/q}$, then there exists an integer $\nu \geq 1$ such that \[f^{\circ q}(z)=z+ C(z-z_0)^{\nu q+1}+\cal O((z-z_0)^{\nu q+2})\quad \text{with}\quad C\neq 0.\]
The integer $\nu$ is called the \textsf{parabolic multiplicity} of $z_0$ as a fixed point of $f$. The map $f^{\circ q}$ has  $\nu q$ attracting petals which are forward invariant and on which the sequence of iterates $f^{\circ nq}$ converges locally uniformly to $z_0$. Those form $\nu$ cycles of attracting petals. 
When $f:\C\to \C$ is an entire map, each cycle of attracting petals must attract the orbit of a critical value or an asymptotic value of $f$. In particular, the map $f$ has at least $\nu$ critical or asymptotic values.

\begin{proposition}\label{prop_boundarystandardmultiple}
In the standard family $(S_{t,a})_{(t,a)\in \cal P_S}$, a point $(t_0,a_0)\in \cal P_S$ is on the boundary of $\cal T_{p/q}$ $\iff$ $S_{t_0,a_0}^{\circ q}-p$ has a multiple fixed point with parabolic multiplicity $\nu =1$.
\end{proposition}

\begin{proof}
($\Rightarrow$) It follows from Lemma \ref{lemma_parabolicNbdry} that when $(t,a)$ is on the boundary of the Arnol' d tongue $\cal T_{p/q}$, then $S_{t,a}^{\circ q}-p$ has a multiple fixed point $x_0\in \R$. 

The map $S_{t,a}$ is an entire mapping with finite order of growth. According to a theorem of Ahlfors \cite{ahlfors}, it has at most finitely many asymptotic values. Since $S_{t,a}$ commutes with translation by $1$, $a$ is an asymptotic value of $S_{t,a}$ if and only if $a+1$ is an asymptotic value of $S_{t,a}$. This shows that $S_{t,a}$ has no asymptotic value. 

Modulo translation by $1$, the map $S_{t,a}$ has only two critical values. Their orbits under iteration of $S_{t,a}$ are symmetric with respect to the real axis. This shows that modulo translation by $1$, the map $S_{t,a}$ has at most $1$ parabolic cycle in $\R$, and either the parabolic multiplicity is $1$ and the attracting direction is contained in $\R$, or $\nu =2$ and the attracting directions are complex conjugate. 

Since the sign of $S_{t,a}^{\circ q}-p$ does not change, there is an attracting direction in $\R$ and so, the parabolic multiplicity is $1$.

($\Leftarrow$) Assume  $S_{t,a}^{\circ q}-p$ has a multiple fixed point $x_0\in \R$ then $(t,a)\in \cal T_{p/q}$. We will show that $(t,a)$ is in the boundary of $\cal T_{p/q}$ by contradiction. If it were not in the boundary of the tongue, then  $S_{t,a}^{\circ q}-p$ would take both positive and negative values. In particular, there would be a point $x_2$ (which a priori might be equal to $x_0$) such that  $S_{t,a}^{\circ q}-p$ takes positive values for $x<x_2$ close to $x_2$ and takes negative values for $x>x_2$ close to $x_2$. Then, $x_2$ would be either an attracting fixed point of $S_{t,a}^{\circ q}-p$ or a multiple fixed point of $S_{t,a}^{\circ q}-p$ with parabolic multiplicity $\nu=2$ and $2$ real attracting directions. The latter is not possible. So, $S_{t,a}$ would have an attracting cycle. Again, this is not possible since this attracting cycle  would have to attract the orbit of a critical value of $S_{t,a}$ whereas the orbit of the critical orbits of $S_{t,a}$ are attracted by the cycle of $S_{t,a}$ containing $x_0$. 
\end{proof}

\section{The regularity of boundary curves}

In this section we try to prove that the boundary of $\cal T_{p/q}$ is the union of two Lipschitz curves. Hall has proved this in \cite{Hall} under the assumption that $\phi$ is $C^1$, using Implicit function theorem. The proof we present here is simpler and we only assume that $\phi$ is Lipschitz. We also show that the irrational tongues are Lipschitz curves. Let us recall the definition of the \textsf{boundary} of the tongues.

\begin{definition}\label{def_bdry}
In the parameter space $\cal P$, the line $a=a_0$ intersects the rational tongue $\cal T_{p/q}$ on an interval $I_{a_0}(p/q)=[\gamma_{p/q}^l(a_0),\gamma_{p/q}^r(a_0)]$ and it intersects the irrational tongue $\cal T_{\alpha}$ for $\alpha\in \R\setminus\Q$ on a singleton set $\{\gamma_{\alpha}(a_0)\}$. Thus we have three type of functions, namely $\gamma_{p/q}^{l}:(a_{\min},a_{\max})\to \R$, $\gamma_{p/q}^r:(a_{\min},a_{max})\to \R$ and $\gamma_{\alpha}:(a_{\min},a_{\max})\to \R$. These functions define the \textsf{boundaries} of the tongues.
\end{definition}

\begin{proposition}\label{prop_lipschbd}
The functions $\gamma_{p/q}^l$ and $\gamma_{p/q}^r: (a_{\min},a_{\max}) \to \R$ are Lipschitz continuous. More precisely
for all $a_0$ and $a_1$ in $(a_{\min},a_{\max})$,
\begin{align*}
&-\max \p \le \frac{\gamma_{p/q}^l(a_1)-\gamma_{p/q}^l(a_0)}{a_1-a_0} \le -\min \p \quad \text{and}\\
&-\max \p \le \frac{\gamma_{p/q}^r(a_1)-\gamma_{p/q}^r(a_0)}{a_1-a_0} \le -\min \p.  
\end{align*}
\end{proposition}

\begin{figure}[htp]
\centering
\includegraphics[width=4in,height=5in]{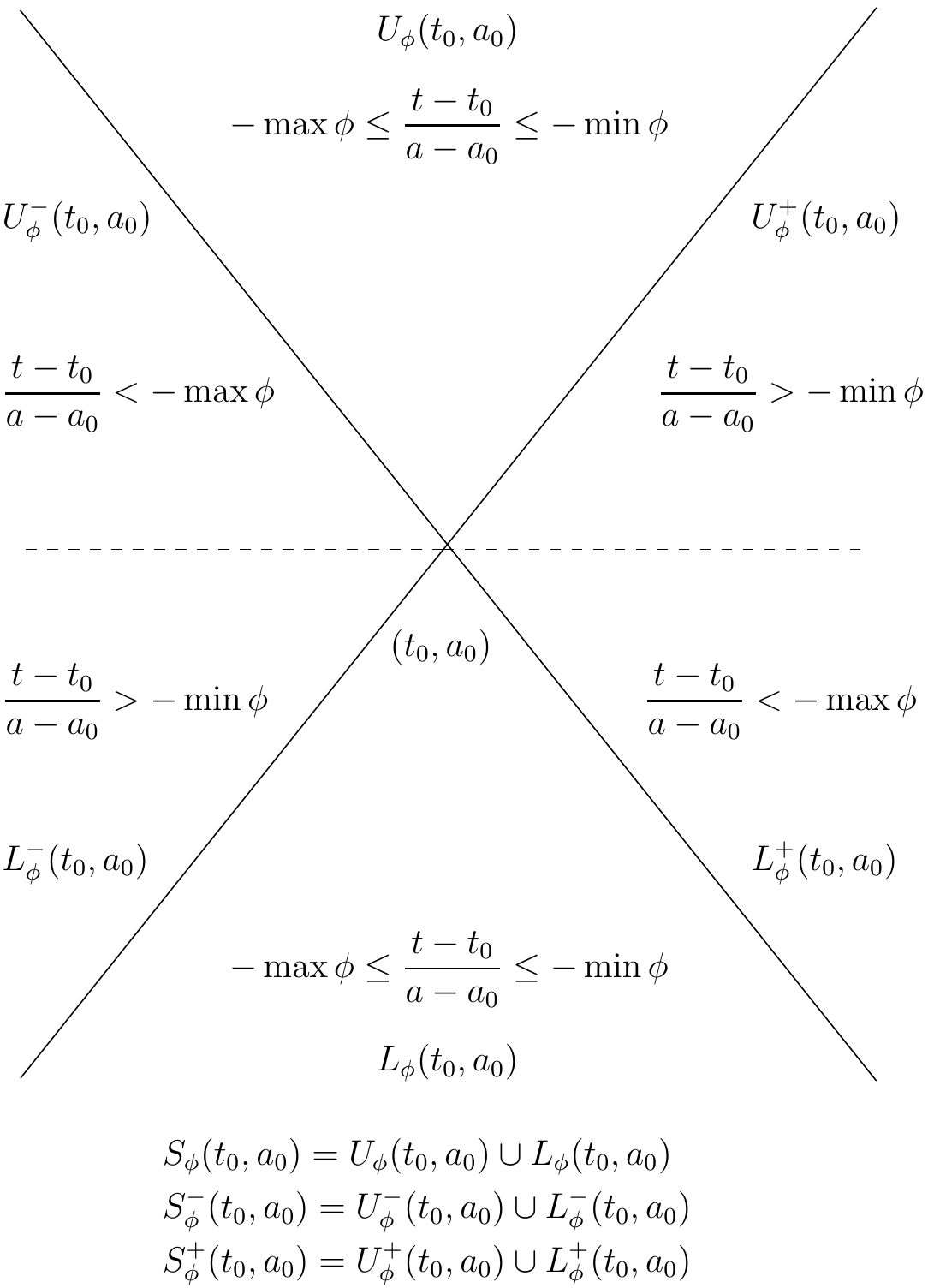}
\caption{Definitions of the region $S_{\p}$ and the sectors $S_{\p}^-,S_{\p}^+$.}
\end{figure}

To prove this we need the following definition and a lemma.

\begin{definition}\label{def_sector}
Let us define some regions of the parameter space $\cal P$ around a base point 
$(t_0,a_0)$ in the following manner.
\begin{align*}
& U_{\p}^-(t_0,a_0)=\bigl\{(t,a)\in \cal P~|~ \frac{t-t_0}{a-a_0} < -\max \p \bigr\} \cap
\bigl\{a \ge a_0 \bigr\} \\
& U_{\p}^+(t_0,a_0)=\bigl\{(t,a)\in \cal P~|~ \frac{t-t_0}{a-a_0} > -\min \p \bigr\} \cap
\bigl\{a \ge a_0 \bigr\} \\
& U_{\p}(t_0,a_0)=\bigl\{(t,a)\in \cal P~|~-\max \p \le \frac{t-t_0}{a-a_0} \le -\min \p \bigr\} \cap \bigl\{a \ge a_0 \bigr\}\\
& L_{\p}^-(t_0,a_0)=\bigl\{(t,a)\in \cal P~|~ \frac{t-t_0}{a-a_0} > -\min \p \bigr\} \cap
\bigl\{a < a_0 \bigr\} \\
& L_{\p}^+(t_0,a_0)=\bigl\{(t,a)\in \cal P~|~\frac{t-t_0}{a-a_0} < -\max \p \bigr\} \cap
\bigl\{a < a_0 \bigr\} \\
& L_{\p}(t_0,a_0)=\bigl\{(t,a)\in \cal P~|~ -\max \p \le \frac{t-t_0}{a-a_0} \le  -\min \p \bigr\} \cap \bigl\{a < a_0 \bigr\}
\end{align*}
Then we define the region $S_{\p}(t_0,a_0)=U_{\p}(t_0,a_0) \cup L_{\p}(t_0,a_0)$, the sectors $S_{\p}^-(t_0,a_0)=U_{\p}^-(t_0,a_0) \cup L_{\p}^-(t_0,a_0)$ and $S_{\p}^+(t_0,a_0)=U_{\p}^+(t_0,a_0) \cup L_{\p}^+(t_0,a_0)$.
\end{definition}

\begin{lemma}\label{lemma_sector}
Suppose that $(t,a)\in \cal P$ . Then
\begin{itemize}
\item $\trans(F_{t,a})\le \trans(F_{t_0,a_0})$ if $(t,a) \in S_{\p}^-(t_0,a_0)$ and 
\item $\trans(F_{t,a})\ge \trans(F_{t_0,a_0})$ if $(t,a) \in S_{\p}^+(t_0,a_0)$.
\end{itemize}
The inequalities are strict if $\trans(F_{t_0,a_0})$ is irrational. 
\end{lemma}

\begin{proof}
Note that $F_{t,a}(x)-F_{t_0,a_0}(x)=(t-t_0)+(a-a_0)\p(x)$. The way the sectors $S_{\p}^-(t_0,a_0)$ and
$S_{\p}^+(t_0,a_0)$ are defined, we see that if $(t,a) \in S_{\p}^-(t_0,a_0)$ then
$F_{t,a}< F_{t_0,a_0}$ and if $(t,a)\in S_{\p}^+(t_0,a_0)$ then $F_{t,a} > F_{t_0,a_0}$. Hence the result follows by Lemma \ref{lemma_transincrease}.
\end{proof}

\begin{proof}[Proof of the Proposition \ref{prop_lipschbd}]
Let us prove the Lipschitz continuity of $\gamma_{p/q}^l$ first. Assume that $t_0=\gamma_{p/q}^l(a_0)$.

First, according to Lemma \ref{lemma_sector}, if $(t_1,a_1) \in S_{\p}^+(t_0,a_0)$ then 
$\trans(F_{t_1,a_1}) \ge p/q= \trans(F_{t_0,a_0})$. 

Second, if $(t_1,a_1) \in S_{\p}^-(t_0,a_0)$ then there is a point $(t_2,a_0)$ in $S_{\p}^+(t_1,a_1)$ with $t_2<t_0$. According to Lemma \ref{lemma_sector}, $\trans(F_{t_1,a_1})\leq \trans(F_{t_2,a_0})$. Since $t_2<t_0$ and $(t_0,a_0)$ is on the left boundary of the Arnol'd tongue $\cal T_{p/q}$, $\trans(F_{t_2,a_0})<\trans(F_{t_0,a_0})=p/q$. Consequently $\trans(F_{t_1,a_1})<p/q$.

It follows that if $t_1=\gamma_{p/q}^l(a_1)$, then  $(t_1,a_1)\in S_{\p}(t_0,a_0)$. This shows the Lipschitz continuity of $\gamma_{p/q}^l$: 
\[-\max \p \le \frac{\gamma_{p/q}^l(a_1)-\gamma_{p/q}^l(a_0)}{a_1-a_0} \le -\min \p. \]
Similarly one can prove that 
\[-\max \p \le \frac{\gamma_{p/q}^r(a_1)-\gamma_{p/q}^r(a_0)}{a_1-a_0} \le -\min \p. \]
These complete the proof of the Lipschitz continuities of $\gamma_{p/q}^{l}$ or $\gamma_{p/q}^r$.
\end{proof}

%\begin{figure}[htp]
%\centering
%\includegraphics[width=3.5in,height=3.5in]{bd_lipschiz.pdf}
%\caption{The bounding curves of the Arnol'd tongue are Lipschitz.}
%\end{figure}

Since $\gamma_{p/q}^l(0)=p/q=\gamma_{p/q}^r(0)$ we have the following corollary.

\begin{corollary}
When $a \rightarrow 0$, we have $\gamma_{p/q}^{l}(a) \rightarrow p/q$ and $\gamma_{p/q}^r(a)\to p/q$.
\end {corollary}

\begin{corollary}
If for a nonzero $a_0$, $a=a_0$ intersects $\cal T_{p/q}$ on a closed interval of positive length then $\cal T_{p/q}$ has non empty interior and thus it is of positive area.
\end {corollary}

\begin{proof}
By construction $\cal T_{p/q}$ is bounded by $\gamma_{p/q}^l$ and $\gamma_{p/q}^r$. Also by assumption $a=a_0$ intersects $\cal T_{p/q}$ on a closed interval say $I_{a_0}(p/q)$, of positive length. Then by the continuity of $\gamma_{p/q}^l$
and $\gamma_{p/q}^r$ there is a non empty open neighbourhood of the interior $I_{a_0}(p/q)^{\circ}$ bounded inside $\cal T_{p/q}$. Which proves that $\cal T_{p/q}$ has non empty interior and it is of positive area.
\end{proof}

Now we show that the irrational tongues are also Lipschitz continuous.

\begin{proposition}
For an irrational $\alpha$, the function $\gamma_{\alpha}:(a_{\min},a_{\max})\to
\R$ is Lipschitz continuous. More precisely, for all $a_0$ and $a_1$ in $(a_{\min}, a_{\max})$,
\[-\max \p \le \frac{\gamma_{\alpha}(a_1)-\gamma_{\alpha}(a_0)}{a_1-a_0} \le -\min \p.  \]
\end{proposition}

\begin{proof}
In this case we shall concentrate on the irrational tongue $\cal T_{\alpha}$ and the function $\gamma_{\alpha}$. The proof of the fact that $\gamma_{\alpha}$ is Lipschitz is similar to that of Proposition \ref{prop_lipschbd}. For
a fixed $a_0\in (a_{\min},a_{\max})$ let $(t_0,a_0)$ be a point on $\gamma_{\alpha}$. From Definition \ref{def_sector} we have the sectors
$S_{\p}^-(t_0,a_0), S_{\p}^+(t_0,a_0)$ and the region $S_{\p}(t_0,a_0)$ defined. According to Lemma \ref{lemma_sector} we see that $\trans(F_{t,a})<\trans(F_{t_0,a_0})=\alpha$ if $(t,a)\in S_{\p}^-(t_0,a_0)$ and $\trans(F_{t,a})>\trans(F_{t_0,a_0})=\alpha$ if
$(t,a)\in S_{\p}^+(t_0,a_0)$. Thus $\trans$ takes the value $\alpha$ in the region $S_{\p}(t_0,a_0)$. Hence
\[-\max \p \le \frac{\gamma_{\alpha}(a)-\gamma_{\alpha}(a_0)}{a-a_0} \le -\min \p. \]
\end{proof}

The Lipschitz continuity of $\gamma_{p/q}^l$ and $\gamma_{p/q}^r$ confirms that $\gamma_{p/q}^l$ and $\gamma_{p/q}^r$ define continuous curves in the parameter space $\cal P$. The same holds true for the curve $\gamma_{\alpha}$. In fact we can prove more when we are in the analytic standard family.

\begin{theorem}\label{theo_bdrystdanalytic1st}
In the standard family $(S_{t,a})_{(t,a)\in \cal P_S}$, the boundary curves of $\cal T_{p/q}$ are analytic functions of $a$ for $a\in (-1/2\pi,1/2\pi)\setminus \{0\}$.
\end{theorem}

\begin{proof}
Suppose that $a_0\ne 0$ is a point in $(-1/2\pi,1/2\pi)\setminus \{0\}$. We shall show that $\gamma_{p/q}^l$ and $\gamma_{p/q}^r$ are analytic around $a_0$. Let's consider $\gamma_{p/q}^l$ first. There exist $t_0$ and $x_0$ such that $(t_0,a_0)\in \gamma_{p/q}^l$ and $(a_0,t_0,x_0)$ satisfies the following equations 
\begin{align}
P(a,t,x) &=S_{t,a}^{\circ q}(x)-x-p=0, \\
Q(a,t,x) &= \frac{\partial}{\partial x}S_{t,a}^{\circ q}(x)-1=0.\label{rel3}
\end{align}
We would try to use the implicit function theorem to obtain that $t$ and $x$ could be expressed as analytic functions of $a$ around $a_0$. 

Let's show that $\ds\frac{\partial}{\partial t}S_{t,a}^{\circ q}(x)\ge 1$ by induction for any triplet $(a,t,x)$. For a fixed value of $a$ define $S(x)=x+a\sin(2\pi x)$ and $S_t(x)=S(x)+t=S_{t,a}(x)$ so that $S_t^{\circ q}(x)=S_{t,a}^{\circ q}(x)$. The statement is true for $q=1$ clearly. Suppose it is true for $q-1$. We note that \[S_t^{\circ q}(x)=S(S_t^{\circ q-1}(x))+t.\] Thus
\[\frac{\partial}{\partial t}S_{t}^{\circ q}(x)=\frac{\partial}{\partial x}S|_{S_t^{\circ q-1}(x)}\cdot \frac{\partial}{\partial t}S_t^{\circ q-1}(x)+1.\]
Since $\ds \frac{\partial}{\partial x}S(x)=\ds\frac{\partial}{\partial x}S_{t,a}(x)>0$ for $S_{t,a}$ being an increasing diffeomorphism and $\ds \frac{\partial}{\partial t}S_t^{\circ q-1}(x) \ge 1$ by induction, it follow that $\ds \frac{\partial}{\partial t}S_t^{\circ q-1}(x)\ge 1$ for any fixed $a$ within its domain. This implies that $\ds \frac{\partial}{\partial t}P(a,t,x)\ge 1$.

We can see that $\ds \frac{\partial}{\partial x}P(a_0,t_0,x_0)=0$ by Equation (\ref{rel3}). By choice 
$x_0$ is a multiple fixed point of $S_{t_0,a_0}^{\circ q}-p$ and according to Proposition \ref{prop_boundarystandardmultiple}, the parabolic multiplicity is $\nu=1$: 
\[S_{t_0,a_0}^{\circ q}(z)-p=z+C(z-z_0)^{2}+\cal O((z-z_0)^{3})\quad \text{with}\quad C\neq 0.\]
Thus $\ds \frac{\partial^2}{\partial x^2}S_{t_0,a_0}^q=2C$ is non zero at $x_0$. Consequently $\ds \frac{\partial Q}{\partial x}$ is non zero at $(a_0,t_0,x_0)$. Therefore the matrix 
${\begin{pmatrix}
\ds\frac{\partial P}{\partial t} & \ds\frac{\partial P}{\partial x} \\
\ds\frac{\partial Q}{\partial t} & \ds\frac{\partial Q}{\partial x}
\end{pmatrix}}$
is invertible when the entries are taken at $(a_0,t_0,x_0)$. Hence by Implicit function theorem $t$ and $x$ can be expressed as analytic function of $a$ around $a_0$.  The same proof holds for $\gamma_{p/q}^r$.
\end{proof}

\section{Angle between the bounding curves of the rational tongue}\label{sec_angle}

In this section we prove that there are lines tangents to these boundary curves of rational tongues at $(p/q,0)$. A slight modification in the definition of the boundaries $\gamma_{p/q}^{l}$ and $\gamma_{p/q}^r$ gives the angle between them. The average of the translates of $\p$ by $p/q$ plays a role here.

\begin{definition}
Let us define the average of the translates of $\p$ by $p/q$ as $\cal A_q(\p)$, i.e.
\[\cal A_q(\p)(x)=\frac{1}{q}\sum_{k=0}^{q-1}\p(x+kp/q).\]
Moreover we set 
\[\cal M(\p)=\int_0^1\p(x)\d x, \quad M_{\cal A}=\max \cal A_q(\p)\quad \text{and}\quad
m_{\cal A}=\min \cal A_q(\p).\]
\end{definition}

\begin{theorem}\label{prop_angle} 
We have the following asymptotic expansions for the boundaries of the tongues.
\begin{itemize}
\item For $\alpha\in \R\setminus \Q$, we have $\gamma_{\alpha}(a)=\alpha-\cal M(\p)a+o(a)$ .
\item For $p/q\in \Q$ if $\cal A_q(\p)\equiv \cal M(\p)$ then $\gamma_{p/q}^{l}(a)=p/q-\cal M(\p)a+o(a)$ and $\gamma_{p/q}^{r}(a)=p/q-\cal M(\p)a+o(a)$.
\item For $p/q\in \Q$, we have \begin{enumerate}
\item $\gamma_{p/q}^-(a)=p/q-M_{\cal A}a+o(a)$, where $\gamma_{p/q}^-(a)=\left\{\begin{array}{ll} \gamma_{p/q}^l(a)\;\text{for}\; a\ge 0,\\
                            \gamma_{p/q}^r(a)\;\text{for}\; a<0. \end{array}\right.$
\item $\gamma_{p/q}^+(a)=p/q-m_{\cal A}a+o(a)$, where $\gamma_{p/q}^+(a)
=\left\{ \begin{array}{ll} \gamma_{p/q}^r(a)\; \text{for}\; a\ge 0,\\
                            \gamma_{p/q}^l(a)\;\text{for}\; a<0. \end{array} \right.$
\end{enumerate}
\end{itemize}
\end{theorem}

The proof depends on how $F_{t,a}^{\circ q}$ behaves near $(p/q,0)$.

\begin{lemma}\label{lemma_angle} 
For small values of $\e$ and $a$,
\[F_{p/q+\e,a}^{\circ q}(x)=x+p+q\e + qa \cal A_q(\p) (x)+a\Psi_{p/q}(x,a,\e)\]
where $\Psi_{p/q}$ is a uniformly continuous function which is $0$ for all $x$ if 
$a=\e=0$.
\end{lemma}

\begin{proof}
Note that
\begin{align*}
F_{p/q+\e,a}^{\circ q}(x) & = x+p + q\e \\
& \quad + a\p\bigl(F_{p/q+\e,a}^{\circ 0}(x)\bigr)+ a\p\bigl(F_{p/q+\e,a}^{\circ 1}(x)\bigl)
+\ldots+ a\p\bigl(F_{p/q+\e,a}^{\circ q-1}(x)\bigr).
\end{align*}
We write $\p(F_{p/q+\e,a}^{\circ k}(x))=\p(x+kp/q)+\psi_k(x,a,\e)$ for each $0 \le k \le q-1$, where $\psi_k$ is defined as $\psi_k(x,a,\e)=\p(F_{p/q+\e,a}^{\circ k}(x))-\p(x+kp/q)$. By definition $\psi_k$ is periodic and continuous, thus it is uniformly continuous and $\psi_k(x,0,0)=0$ for all $x$.

Therefore
\begin{align*}
F_{p/q+\e,a}^{\circ q}(x)& =x+p+q\e + a\sum_{k=0}^ {q-1}\p(x+ kp/q) +
   a\sum_{k=0}^{q-1}\psi_k(x,a,\e) \\
& = x+p+q\e + aq\cal A_q(\p) (x)+ a\Psi_{p/q}(x,a,\e);
\end{align*}
where $\Psi_{p/q}(x,a,\e)=\ds \sum_{ k=0}^{q-1}\psi_k(x,a,\e)$ and $\Psi_{p/q}(x,0,0)=
\ds \sum_{k=0}^{q-1}\psi_k(x,0,0)=0$.
\end{proof}

\begin{proof}
[Proof of Proposition \ref{prop_angle}]
\underline{Case (i)} First we look at the boundaries of the rational tongue $\cal T_{p/q}$ assuming their modified definition. We consider a particular case here, when we are approaching the left boundary curve $\gamma_{p/q}^-$ from above i.e. $a>0$. Define 
\[\lambda(a)=\displaystyle\frac{\gamma_{p/q}^-(a)-p/q}{a}.\] It is enough to show that $\lambda(a) \rightarrow
-M_{\cal A}$ as $a \rightarrow 0$ with $a>0$. In other words for a given $\delta >0$ we have to show that $|\lambda(a)+M_{\cal A}|\le \delta$ as $a \rightarrow 0$ with $a>0$.

By the continuity of $\Psi_{p/q}$ we can choose $r_1>0$ so that if $|a|\le r_1$ and $|\e|\le r_1$ then
$\bigl|\displaystyle\frac{1}{q}\Psi_{p/q}(x,a,\e)\bigr| \le \delta$ for any $x$. Choose $r_2>0$ using the continuity of $\gamma_{p/q}^-$ such that if
$|a|<r_2$ then $|\gamma_{p/q}^-(a)-p/q|<r_1$. Now fix $r=\min(r_1,r_2)$ and take $0<a<r$.

For $\e(a)=a.\lambda(a)$, we are on $\gamma_{p/q}^-$ i.e. we are are on the tongue. Which implies that there is an $x_a$ such that $F_{p/q+\e(a),a}^{\circ q}(x_a) = x_a+p$. By Lemma \ref{lemma_angle} we see that
\[\lambda(a)+\cal A_q(\p) (x_a)+\frac{1}{q}\Psi_{p/q}(x_a,a,\e(a))=0\]
\[\Rightarrow \lambda(a)=-\cal A_q(\p) (x_a)-\frac{1}{q}\Psi_{p/q}(x_a,a,\e(a)) \ge -M_{\cal A}-\delta\]
\[\Rightarrow \lambda(a)+M_{\cal A}\ge-\delta.\]

For the other inequality assume that $x_0$ is a point such that $\cal A_q(\p) (x_0)=M_{\cal A}$. Since we are considering the left boundary of
the tongue $\cal T_{p/q}$, we note that the graph of the function $F_{p/q+\e(a),a}^{\circ q}$ lies below the graph of the function $y=x+p$. Thus
\[F_{p/q+\e(a),a}^{\circ q}(x_0)\le x_0+p\]
\[\Rightarrow \lambda(a)+\cal A_q(\p) (x_0)+\frac{1}{q}\Psi_{p/q}(x_0,a,\e(a))\le0 \;\ (\mbox {By  Lemma \ref{lemma_angle}})\]
\[\Rightarrow \lambda(a)+M_{\cal A}\le -\frac{1}{q}\Psi_{p/q}(x_0,a,\e(a))\le \delta.\]

The other cases follow similarly. When the function $\cal A_q(\p)$ is constant, then $m_{\cal A} = M_{\cal A}=\cal M(\p)$. So,
\[\gamma_{p/q}^{l}(a)=p/q-\cal M(\p)a+o(a)\quad \text{and}\quad 
\gamma_{p/q}^{r}(a)=p/q-\cal M(\p)a+o(a).\]

\bigskip 

\noindent
\underline{Case (ii)} 
According to a Theorem of Herman (\cite{Hershort}, see also Theorem 4.2 in \cite{widths} for a simpler version) for all $s\in \R$, 
\[\trans(F_{\alpha+sa, a}) = \alpha + a\cdot (s+\cal M_\p)+ o(a).\]
As a consequence, if $s>-\cal M_\p$ and $a$ is sufficiently close to $0$, then $\trans(F_{\alpha+sa, a}) >\alpha$ . If $s<-\cal M_\p$ and $a$ is sufficiently close to $0$, then $\trans(F_{\alpha+sa, a}) <\alpha$.
Hence \[\gamma_{\alpha}(a)=\alpha-\cal M(\p)a+o(a).\] \end{proof}

A small calculation gives the following corollary.

\begin{corollary}
The angle between the left and right bounding curves of $\cal T_{p/q}$ is
\[ \arctan \frac {(M_{\cal A}-m_{\cal A})(1+m_{\cal A}M_{\cal A})}{(m_{\cal A}M_{\cal A})^2}. \]
\end {corollary}

\begin{remark}
For a fixed $q$ the angle between the two bounding curves of $\cal T_{p/q}$ remains same for all $p$ coprime to $q$.
\end{remark}

\begin{example}
If $\cal A_q(\p) $ is non constant for some well chosen $\p$ then we have a non zero angle between the boundaries of the $p/q$ tongue. Checking the Fourier expansion of $\p$ we note that if all the Fourier coefficients of $\p$ are non zero then we would have angle between the boundaries of each rational tongue. Such an example is \[\p(x)=e^{\cos (2\pi x)}\sin(\sin(2\pi x))=\displaystyle {\sum_{n\ge0} \frac{\sin(2\pi nx)}{n!}}.\] 
For this choice of $\p$ we have 
\[\cal A_q(\p) (x)=\displaystyle{\sum_{m\ge0} \frac{\sin(2\pi mqx)}{m!}};\] which is non constant for every $q$, thus we have non trivial angle between the boundaries of each rational tongue $\cal T_{p/q}$. We call this the \textsf{Angle Family}.
$\blacklozenge$
\end{example}

\begin{figure}[htp]
\centering
\includegraphics[width=5in,height=2.5in]{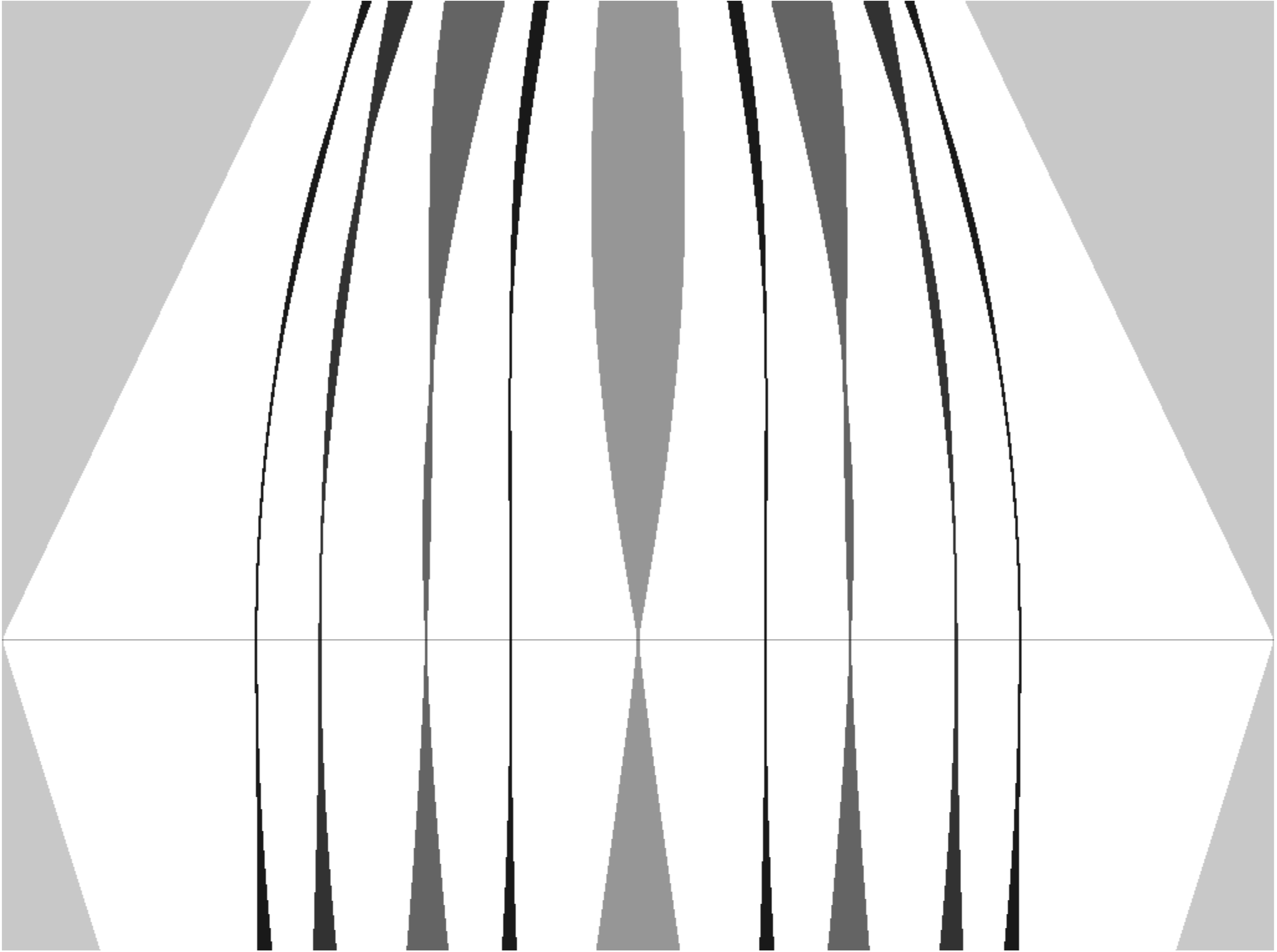}
\caption{Arnol'd Tongues of the Angle family}
\end{figure}

\section{Order of contact of the boundaries of the rational tongue}\label{sec_ooc}

In the previous section we discussed about the possible angle between the boundaries of the rational tongue in a two parameter family. In many known example the function $\cal A_q(\phi)$ is constant and consequently the angle between these boundaries is zero. This is the case for the standard family $S_{t,a}$ for any rational tongue $\cal T_{p/q}$ with $q\geq 2$. In this situation it is interesting to study the \textsf{order of contacts of the boundaries of the rational tongue}. 

\begin{definition}
Assume that for a certain two parameter family $F_{t,a}:\R\to \R$, the boundaries of the rational tongue $\cal T_{p/q}$ are functions of $a$ and the function $\cal A_q(\p) $ is constant for $q\ge 2$. Then we say that $k$ is the order of contact of the boundaries of $\cal T_{p/q}$ for $q\ge 2$ if 
\[|\gamma_{p/q}^r(a)-\gamma_{p/q}^l(a)|\underset{a\to 0}=\cal O(a^k)\quad\text{but}\quad |\gamma_{p/q}^r(a)-\gamma_{p/q}^l(a)|\underset{a\to 0}\ne\cal O(a^{k+1}).\]

%if $\gamma_{p/q}^r(a)-\gamma_{p/q}^l(a)$ is in $\cal O(a^k)$ but not in $\cal O(a^{k+1})$ as $a\to 0$. 
\end{definition}

%\begin{remark}
%It will be convenient to modify the definition of the boundaries when the order of contact is even. 
%For the rest of this chapter we assume that $\gamma_{p/q}^-=\gamma_{p/q}^l$ and $\gamma_{p/q}^+=\gamma_{p/q}^r$ for %$a\ge 0$. For  $a<0$ we set $\gamma_{p/q}^-=\gamma_{p/q}^l$ and $\gamma_{p/q}^+=\gamma_{p/q}^r$ if the order of %contact is odd and  $\gamma_{p/q}^-=\gamma_{p/q}^r$ and $\gamma_{p/q}^+=\gamma_{p/q}^l$ is the order of contact is %odd.
%\end{remark}

The order of contact of the boundaries of the rational tongues in the standard family is a known fact. Arnol'd has shown that the order of contact of the boundaries of $\cal T_{p/q}$ is at least $q$ in \cite{Arnoldcontact} when Broer, S\'imo and Tatjer have proved that the order of contact of the boundaries of $\cal T_{p/q}$ is exactly $q$ in \cite{broeretal}. It is not known that this phenomenon is related to the fact that the parabolic multiplicity of $0$ as a fixed point of the map $z\mapsto e^{i2\pi p/q}ze^{\pi z}$ is equal to $1$. 

The positive order of contact in the standard family $(S_{t,a})_{(t,a)\in\cal P_S}$ is due to some of its properties: the map $\Pi_a:x\mapsto ae^{i2\pi x}$ semiconjugates $S_{t,a}:\R\to \R$ to $s_{t,a}:C_{|a|}\to C_{|a|}$ where $s_{t,a}:\C^*\to \C^*$ is defined by 
\[s_{t,a}(z)=e^{i2\pi t}ze^{\pi (z-a^2/z)},\] 
and as $a\to 0$, the maps $s_{t,a}$ converge uniformly to $s_t:z\mapsto e^{i2\pi t}ze^{\pi z}$ on compact subsets of $\C^*$. To study the order of contacts in similar families like the standard family, we introduce the notion of  \textsf{admissible} and \textsf{guided} family. Our aim would be to show that the order of contact of the boundaries of $\cal T_{p/q}$ is a multiple of $q$ in an admissible guided family. 

%Another family of interest is the Blaschke fraction family which we shall introduce shortly. It would be shown that the Blaschke fraction family is admissible and guided. And the order of contact of the boundaries of $\cal T_{p/q}$ is also exactly $q$ in this case.

\begin{definition}
Let $I$ and $J$ be open intervals of $\mathbb R$ such that $p/q\in I$ and $0 \in J$.
A family $(F_{t,a}:\mathbb R \rightarrow \mathbb R)_{(t,a) \in I \times J}$ is \textsf{admissible} if
\begin{itemize}
\item The map $J \times I \times \mathbb R \ni (a,t,x) \mapsto F_{t,a}(x) \in \mathbb R$ is $\mathbb R$-analytic.
\item For all $(t,a)\in I\times J$, we have $F_{t,a}(x+1) = F_{t,a}(x) +1$.
\item For all $t\in I$, the map $F_{t,0}$ is the translation $T_{\theta(t)}$.
\end{itemize}

An admissible family $(F_{t,a}:\R\to \R)_{(t,a)\in I\times J}$ is \textsf{guided} by a family of holomorphic maps $(f_t :D_r\to \C)_{t\in I}$ if there exists an analytic family of holomorphic maps $(f_{t,a}:A_{a,r}\to \C^*)_{(t,a)\in I\times J_r}$ with
$J_r := J\cap (-r,r)$ such that
\begin{itemize}
\item For all $(t,a)\in I\times J_r$, we have $f_{t,a}\circ \Pi_a = \Pi_a\circ F_{t,a}$ on 
$\R$.
\item $f_{t,0} = f_t $ on $D_r\setminus\{0\}$.
\end{itemize}
\end{definition}

From the definition of the admissible family we see that
\[F_{t,a}(x)=x+\theta(t)+ a \xi(a,t,x); \]
where $\xi$ is an analytic function defined on a neighbourhood of $J \times I \times \R$  which is also periodic of period 1 in $x$. It is {guided} by a holomorphic family $(f_t:D_r\to \C)_{t\in I}$ if there is an analytic family of maps $(f_{t,a})_{(t,a)\in I\times J_r}$ defined on the annulus $A_{a,r}$ to $\C^*$ such that
\begin{itemize} 
\item $x\mapsto ae^{2i\pi x}$ semiconjugates $F_{t,a}$ to $f_{t,a}$.

\item $f_{t,a} \underset{a \to 0}\longrightarrow f_t$ locally uniformly on $D(0,r)\setminus\{0\}$.
\end{itemize}

\begin{example}
It is easy to see that the standard family $(S_{t,a}:x\mapsto x+t+a\sin(2\pi x))_{(t,a)\in I\times J}$ is admissible. Here $I=\R$ and $J=(-1/2\pi,1/2\pi)$. 
For $J_{r}=J\cap(-r,r)$ we try to define another family $(s_{t,a} : A_{a,r}\to \C^*)_{(t,a) \in I\times J_r}$. Let's define $s_{t,a}$ on $|z|=|a|$ first. 
\[s_{t,a} (ae^{i2 \pi x}) =\displaystyle ae^{i2 \pi  (x+t)}e^{ i2\pi a\ds [\frac{e^{i2 \pi  x}}{2i}- \frac{e^{-i2 \pi  x}}{2i}]}
= \displaystyle {ae^{i2 \pi  (x+t)}e^{a\pi{e^{i2 \pi  x}}-{a\pi e^{-i2 \pi  x}}}}.\]
Assuming $ae^{i2 \pi  x}=z$ we see that,
\[s_{t,a} (z)= e^{ i2 \pi  t}ze^{\ds \pi z -\frac{a^2\pi}{z}}.\]
This gives a well defined family $s_{t,a}:A_{a,r}\to \C^*$ such that  
$s_{t,a} \circ \Pi_a(x)=\Pi_a \circ S_{t,a} (x)$ for any $(t,a) \in I\times J_r$. This implies that as $a\to 0$ we see that $s_{t,a}  \to s_t : z \mapsto e^{i2 \pi  t}ze^{\pi z}$. Consequently the standard family is guided by $(s_t :D_r\to \C)_{t \in I }$ such that $s_t (z)=e^{i2 \pi  t}ze^{\pi z}$.
$\blacklozenge$
\end{example}

\begin{example}\label{example_Bla}
Another interesting family is the \textsf{Blaschke fraction family}. In this case the family of circle homeomorphisms is given by 
\[\widehat{B}_{t,a}: z\mapsto e^{i2\pi t}z\frac{1-az}{1-a/z}\] when we take $z$ on the unit circle. The parameter space for this family is $\cal P_B=\{(t,a)~|~ t\in \R, -1/\sqrt{2}<a < 1/\sqrt{2}\}$. Here we take $I=\R$ and $J=(-1/\sqrt 2,1/\sqrt 2)$. One could argue that the family is induced by a family of homeomorphisms of the real line given by
\[B_{t,a}:x\mapsto x+t+2\arctan {\frac{a\sin(2\pi x)}{1-a\cos(2\pi x)}}.\] 
The map $z\mapsto az=w$ semiconjugates $\widehat{B}_{t,a}(z)$ to the rational function  $b_{t,a}:w\mapsto e^{i2\pi t}w\ds\frac{1-w}{1-a^2/w}$. Also $b_{t,a}$ tends to the quadratic family $b_t:w\mapsto e^{i2\pi t}w(1-w)$ uniformly on the compact subsets $D_r\setminus\{0\}$ as $a\to 0$. Thus $(B_{t,a})_{(t,a)\in I\times J}$ is an admissible family guided by $(b_t)_{t\in I}$.
$\blacklozenge$
\end{example}

In the following discussions we assume that $(F_{t,a}:\R\to \R)_{(t,a)\in I\times J}$ is an admissible family guided by a holomorphic family $(f_t :D_r\to \C)_{t \in I}$. We would try to understand the properties of these families and characterise them according to some properties. First we show that for all $t\in I$, the map $f_t $ has one indifferent fixed
point at $z=0$.

\begin{lemma} \label{lemma_adgdd1}
For all $t\in I$, we have
\[f_t (0)=0\quad \text{and}\quad f_t '(0) = e^{i2\pi \theta(t)}.\]
\end{lemma}

\begin{proof}
Let $(g_{t,a}:\mathbb S^1\to \mathbb S^1)_{(t,a)\in I\times J}$ be the family of maps defined by the following relation \[g_{t,a}\circ \Pi = \Pi \circ F_{t,a};\]
so that for all $z\in \mathbb S^1$, we have $f_{t,a}(az) = a g_{t,a}(z)$.

Looking at the Laurent series coefficients $d_k$ of $f_{t,a}$ and $g_{t,a}$ we see that
\[d_k(f_{t,a})\underset{a\to 0}\longrightarrow d_k(f_t )\quad\text{and}\quad d_k(g_{t,a})\underset{a\to 0}\longrightarrow d_k(R_{\theta(t)});\] 
with $d_k(R_{\theta(t)})=0$
for $k\neq 1$ and $d_k(R_{\theta(t)})=e^{i2\pi \theta(t)}$
for $k=1$. Moreover
\begin{align*} d_k(f_{t,a}) =
\frac{1}{i2\pi}\int_{C(0,|a|)} \frac{f_{t,a}(z)}{z^{k+1}}\ \d z &=
\frac{1}{i2\pi}\int_{C(0,1)}
\frac{f_{t,a}(aw)}{(aw)^{k+1}}\ \d (aw) \\
&= \frac{1}{ a^{k-1}}\cdot\frac{1}{i2\pi}\int_{C(0,1)}
\frac{g_{t,a}(w)}{w^{k+1}}\ \d w \\
&=\frac{1}{ a^{k-1}}\cdot d_k(g_{t,a}).
\end{align*}
We obtain the result by taking the limit when $a$ tends to $0$ for $k=0$ and $k=1$.
\end{proof}

\begin{remark}
Note that if $(F_{t,a}:\R\to \R)_{(t,a)\in I\times J}$ is an admissible family guided by a holomorphic family
$(f_t :D_r\to \C)_{t\in I}$, then for all $k\geq 0$ there exists $r'>0$ such that the family
$(F_{t,a}^{\circ k}:\R\to \R)_{(t,a)\in I\times J}$ becomes admissible and guided by the holomorphic family
$(f_t^{\circ k}:D_{r'}\to \C)_{t \in I}$. It is sufficient to choose $r'\in (0,r)$ so that
$f_t^{\circ k}$ is defined on $D_{r'}$ for all $t\in I$.
\end{remark}

\begin{proposition}\label{prop_adgddch}
Suppose $(F_{t,a}:\R \rightarrow \R)_{(t,a)\in I \times J}$ is guided and admissible and 
\[F_{t,a}(x)=x+\theta(t)+\sum_{n\ge 1} \Xi_{n}(t,x)a^n,\] 
then $\Xi_{n}(t,x)$ is a trigonometric polynomial of degree $\le n$ in $x$. In other words
\[\Xi_{n}(t,x)=\displaystyle\sum_{|k|\le n}c_{n,k}(t)e^{i2\pi kx}\] 
where $c_{n,k}(t)$ is the $k$-th Fourier coefficient of $\Xi_n(t,x)$.
\end{proposition}

\begin{proof}
We have to show that the Fourier expansion of $\Xi_n(t,x)$ does not contain any non zero terms outside the $-n$-th and $n$-th terms. We write \[F_{t,a}(x)=G_{t,a}(x)+a^{n_0}H(a,t,x),\] where $n_0$ is the least $n$ such that $\Xi_n(t,x)$ is not a trigonometric
polynomial of degree $\le n$ in $x$; i.e. there is a $k>n$ and $c_{n,k}(t)\neq 0$.

The first part of the proof contains showing that $(G_{t,a}:\R \to \R)_{(t,a)\in I \times J}$ is guided and admissible. Note that $G_{t,a}(x)=x+\theta(t)+\displaystyle\sum_{1\le n < n_0} \Xi_n(t,x)a^n$, where
$\Xi_n(t,x)=\displaystyle\sum_{|k|\le n}c_{n,k}(t)e^{i2\pi kx}$. The way $G_{t,a}$ is chosen it follows
that it is admissible. We claim that it is guided.

Define $g_{t,a}:\C^*\to \C^*$ as
\[g_{t,a}(z)= ze^{i2\pi \theta(t)}\displaystyle\prod_{1 \le n <n_0}e^{i2\pi a^n{\displaystyle\sum_{|k|\le n}}c_{n,k}(t)(z/a)^k}
            = ze^{i2\pi \theta(t)}\displaystyle\prod_{1 \le n <n_0}e^{i2\pi{\displaystyle\sum_{|k|\le n}}a^{n-k}c_{n,k}(t)z^k}.\]
This implies that             
$g_{t,0}(z)=ze^{i2\pi \theta(t)}{\displaystyle\prod_{1 \le n <n_0}}e^{i2\pi c_{n,n}(t)z^n}$. We define $ g_{t}: \C \to \C$ by $g_t=g_{t,0}$.  
By construction, the family $(g_{t,a}:A_{a,r}\to \C^*)_{(t,a)\in I \times J_r}$ is analytic where $J_r := J\cap (-r,r)$. In addition we also have $g_{t,a}\circ \Pi_a = \Pi_a\circ G_{t,a}$ on $\R$. So $(G_{t,a})_{(t,a)\in I\times J}$ is guided by $(g_t:D_r \to \C)_{t\in I}$.

As $F_{t,a}$ is guided and admissible, for $x\in \R$ we have
\begin{align}
f_{t,a}(ae^{i2\pi x}) &= ae^{i2\pi F_{t,a}(x)}\\
&= ae^{i2\pi(G_{t,a}(x)+a^{n_0}H(a,t,x))}\\
&= g_{t,a}(x)e^{i2\pi a^{n_0}H(a,t,x)}\\
\Rightarrow \frac{f_{t,a}(ae^{i2\pi x})}{g_{t,a}(ae^{i2\pi x})}&=e^{i2\pi a^{n_0}H(a,t,x)}\label{rel1}
\end{align}
Since $g_{t,a}:\C^* \to \C^*$ does not have any zeros on $\C^*$, $\displaystyle\frac{f_{t,a}}{g_{t,a}}$ defines a
holomorphic function on $A_{a,r} \subset \C^* $. Similarly $\displaystyle\frac{f_{t}}{g_{t}}$ is holomorphic on
$D_r\setminus\{0\}$; in fact it is holomorphic on $D_r$ by Lemma \ref{lemma_adgdd1}. Which means $\Bigl(\displaystyle\frac{f_{t,a}}{g_{t,a}}:A_{a,r} \to \C^*\Bigr)_{(t,a)\in I_r\times J}$ and
$\Bigl(\displaystyle\frac{f_{t}}{g_{t}}:D_r \to \C\Bigr)_{(t \in I)}$ are families of analytic maps. The next observation is that as $a\rightarrow 0, \displaystyle\frac{f_{t,a}(z)}{g_{t,a}(z)} \rightarrow
\displaystyle\frac{f_{t}(z)}{g_{t}(z)}$ uniformly on compact subsets of $A_{a,r}$. Since $A_{a,r}\rightarrow D_r\setminus \{0\}$
with $a\rightarrow 0$, the above convergence is true on $D_r\setminus\{0\}$.

Define two new functions $L_{t,a}$ and $K_{t,a}$ when $(t,a) \in I\times J$ and $x\in \R$ by the following equations as follows.
\begin{equation}\label{rel2}
L_{t,a}(x)=e^{i2\pi(F_{t,a}(x)-G_{t,a}(x))}= e^{i2\pi a^{n_0}H(a,t,x)}= 1+a^{n_0}K_{t,a}(x)
\end{equation}
By choice $K_{t,a}(x)=\displaystyle\sum_{j\ge 1} \frac{(i2\pi)^j}{j!}a^{n_0(j-1)}(H(a,t,x))^j$, which is analytic on a neighbourhood of $J\times I\times \R$. And it is also evident that $L_{t,a}(x)$ is analytic on the same domain. Moreover $K_{t,0}(x)=H(0,t,x)$. As
$a\rightarrow 0$, $K_{t,a}(x)\rightarrow K_{t,0}(x)$ uniformly on compact subsets of $D_r\setminus\{0\}$.
Take $a$ small enough such that the circle $|z|=r/2$ is inside $A_{a,r}$. 
Assume that $k\in \N$, looking at the Fourier coefficients of $L_{t,a}$ and $K_{t,a}$ we have
\[c_k(L_{t,a})\underset{\ref{rel2}}=a^{n_0}c_k(K_{t,a}).\]
Moreover
\begin{align*}
c_k(L_{t,a})&=\int_0^1 L_{t,a}(x)e^{-i2\pi kx}\d x \\
&\underset{\ref{rel1}}=\int_0^1 \frac{f_{t,a}(ae^{i2\pi x})}{g_{t,a}(ae^{i2\pi x})}e^{-i2\pi kx}\d x \\
&=\frac{1}{i2\pi}\int_{|z|=a} \frac{f_{t,a}(z)}{g_{t,a}(z)}\frac{a^k}{z^{k+1}}\d z\,\ (ae^{i2\pi x}=z)\\
&=\frac{a^k}{i2\pi}\int_{|z|=r/2}\frac{f_{t,a}(z)}{g_{t,a}(z)z^{k+1}}\d z\\
&=\frac{a^k}{i2\pi}d_k\Bigl(\frac{f_{t,a}}{g_{t,a}}\Bigr).
\end{align*}
\[\therefore c_k(K_{t,a})=\frac{a^{(k-n_0)}}{i2\pi} d_k\Bigl(\frac{f_{t,a}}{g_{t,a}}\Bigr).\]

With $a\rightarrow 0, K_{t,a}(x)\rightarrow K_{t,0}(x)=H(0,t,x)$ and $\displaystyle\frac{f_{t,a}(z)}{g_{t,a}(z)} \rightarrow
\displaystyle\frac{f_{t}(z)}{g_{t}(z)}$ by uniform continuity on $|z|=r/2$. Which implies that $c_k(K_{t,0})=0$ if $k>n_0$. Since $F_{t,a}(x)\in \R$ for $(a,t,x)\in J\times I\times \R$, we note that $\Xi_{n_0}(t,x)\in \R$ for $(t,x)\in I\times\R$; and
$c_{n_0,-k}(t)=\overline{c_{n_0,k}(t)}$ for $k>0$. Thus $c_k(K_{t,0})=0$ if $|k|>n_0$.
Hence we arrive at a contradiction. This completes the proof.
\end{proof}

The following is the analytic characterization of the admissible and guided family of analytic circle diffeomorphisms.  

\begin{theorem}\label{theo_characteranalytic}
Suppose that $(F_{t,a}:\R \rightarrow \R)_{(t,a)\in I \times J}$ is an analytic family and 
\[F_{t,a}(x)=x+\theta(t)+\sum_{n\ge 1} \Xi_{n}(t,x)a^n.\] 
The family $(F_{t,a})_{(t,a)\in I \times J}$ is admissible and guided 
$\iff$ for any $n\in \N$, $\Xi_{n}(t,x)$ is a trigonometric polynomial of degree $\le n$ in 
$x$, in other words \[\Xi_{n}(t,x)=\displaystyle\sum_{|k|\le n}c_{n,k}(t)e^{i2\pi kx};\] 
where $c_{n,k}(t)$ is the $k$-th Fourier coefficient of $\Xi_n(t,x)$.
\end{theorem}

\begin{proof}
($\Rightarrow$) This part is done in Proposition \ref{prop_adgddch}.

($\Leftarrow$) It is evident that the family is admissible. We have to show that it is guided. Let's define a complex valued function $f_{t,a}$ on $\{|z|=|a|\}$ by the following relation
\[f_{t,a}(ae^{i2\pi x})=ae^{i2\pi(x+\theta(t))}\prod_{n\ge 1}e^{i2\pi \ds \sum_{|k|\le n} c_{n,k}(t)a^ne^{i2\pi kx}}.\] It is well defined on the circle $|z|=|a|$. If we take $ae^{i2\pi x}=z$ in the previous relation then we have 
\[f_{t,a}(z)=ze^{i2\pi \theta(t)}\prod_{n\ge 1}e^{i2\pi \ds\sum_{|k|\le n}a^{n-k}c_{n,k}(t)z^k}.\] 
Note that $f_{t,a}:\C^*\to \C^*$ gives an analytic map. We set $f_t:=f_{t,0}$. Thus 
\[f_t(z)=ze^{i2\pi \theta(t)}\prod_{n\ge 1}e^{i2\pi c_{n,n}(t)z^n}.\] By construction the analytic families $(f_{t,a}:A_{a,r}\to \C^*)_{(t,a)\in I\times J_r}$ and $(f_t:D_r\to \C)_{t\in I}$ satisfy the conditions that $(F_{t,a})_{(t,a)\in I\times J}$ is guided.
\end{proof}

\begin{definition}
We define an analytic function $\Phi$ on a neighbourhood of $J \times I\times \R$ in the following way.
\[\Phi(a,t,x)\eqdef x+p-F_{t,a}^{\circ q}(x).\]
\end{definition}

Our next target is to study the function $\Phi$ so that we can express $t$ as a power series of $a$, for $(t,a)$ in the boundary of $\cal T_{p/q}$.

\begin{lemma}\label{lemma_induction}
Suppose $(F_{t,a}:\mathbb R \rightarrow \mathbb R)_{(t,a)\in I \times J}$ is an admissible family guided by a holomorphic
family $(f_t:D_r\to \C)_{t \in I}$. Assume $t_0\in I, \theta(t_0)=p/q$ and $\theta'(t_0)=1$. For $n\in [1,q]$, there exists
\begin{itemize}
\item real numbers $c_{n-1}$, $M_n$,

\item an analytic function $\Phi_n$ in a neighbourhood of $\{0\}\times [-M_n,M_n]\times \R$ and

\item a function $\phi_n:\R\to \R$
\end{itemize}
such that
\begin{itemize}
\item  if $(t,a)\in {\mathcal T}_{p/q}$, then  $t = c_0 + ac_1 + \cdots + a^{n-1}c_{n-1}+ a^n \tau_n$ with $|\tau_n|\leq M_n$,

\item $\Phi(a,c_0 + ac_1 + \cdots + a^{n-1}c_{n-1}+ a^n \tau_n,x) = a^n\Phi_n(a,\tau_n,x)$ and

\item $\Phi_n(0,\tau_n,x)= \phi_n(x)-q\tau_n$.
\end{itemize}
\end{lemma}

Before proving this lemma let us prove the following proposition which could arise in a more general situation starting with
just an admissible family.

\begin{proposition}\label{prop_periodicpart}
Suppose $(F_{t,a}:x \mapsto x+\theta(t)+a\xi(a,t,x))_{(t,a)\in I \times J}$ is an admissible family and $\phi_n:\R\to \R$,
$\Psi_n:J\times [-M_n,M_n]\times \R \to \R$ are analytic for a constant $M_n$.
Also assume that $(t_0,0)\in \mathcal T_{p/q}, \theta(t_0)=p/q$ and if $(t,a)\in {\mathcal T}_{p/q}$ then
\begin{itemize}
\item $t = c_0 + ac_1 + \cdots + a^{n-1}c_{n-1}+ a^n \tau_n$ for constants $c_0=t_0,c_1,\cdots,c_{n-1}$ with $|\tau_n|\leq M_n$,

\item $F_{t,a}^{\circ q}(x)=x+p-\Phi(a,t,x)=x+p-a^n(\phi_n(x)-q\tau_n)-a^{n+1}\Psi_n(a,\tau_n,x).$
\end{itemize}
In this case $\phi_n$ is $p/q$ periodic.
\end{proposition}

\begin{proof}
Let us calculate $F_{t,a}^{\circ q+1}$ in two different ways and compare. Note that for $(t,a)\in \mathcal T_{p/q}$ we have
\[ F_{t,a}^{\circ q}(x)=x+p-\Phi(a,t,x)=x+p-a^n(\phi_n(x)-q\tau_n)-a^{n+1}\Psi_n(a,\tau_n,x).\]
And we assumed in the beginning that $F_{t,a}(x)=x+\theta(t)+ a \xi(a,t,x) $. Then
\begin{align*}
F_{t,a}^{\circ q+1}(x)=F_{t,a}^{\circ q}\circ F_{t,a}(x)&=x+\theta(t)+a\xi(a,t,x)+p-a^n\phi_n(x+\theta(t)+a\xi(a,t,x))+a^nq\tau_n\\
& -a^{n+1}\Psi_n(a,\tau_n,x+\theta(t)+a\xi(a,t,x)).
\end{align*}
Similarly
\begin{align*}
F_{t,a}^{\circ q+1}(x)=F_{t,a}\circ F_{t,a}^{\circ q}(x)&=x+p-a^n\phi_n(x)+a^nq\tau_n+\theta(t)+a\xi(a,t,F_{t,a}^{\circ q}(x))\\
& -a^{n+1}\Psi_n(a,\tau_n,x).
\end{align*}
As $\xi$ is 1-periodic with respect to $x$ we see that
\[\xi(a,t,F_{t,a}^{\circ q}(x))=\xi(a,t,x+p+\mathcal O(a^n))=\xi(a,t,x)+\mathcal O(a^n).\]
Comparing the above two we see that
\begin{align*}
a^n\bigl(\phi_n(x)-\phi_n(x+\theta(t)+a\xi(a,t,x))\bigr)&=a^{n+1}\bigl(\Psi_n(a,\tau_n,F_{t,a}(x))-\Psi_n(a,\tau_n,x)\bigr)\\
&+a\xi(a,t,F_{t,a}^{\circ q}(x))-a\xi(a,t,x)\\
&=\mathcal O(a^{n+1}).
\end{align*}
Dividing two sides by $a^n$ and taking the limit as $a \to 0$ and $t \to t_0$ we see $\phi_n(x)=\phi_n(x+p/q)$.

\end{proof}

\begin{proof}[Proof of Lemma \ref{lemma_induction}]
We prove this lemma by induction. To begin the arguments we prove the base case $n=1$ first.
%The function $(t,a)\mapsto \Rot(F_{t,a})$ is continuous. If we choose $a$ sufficiently close to $0$, the map                                                                                                    
%$(t,a)\mapsto \Rot(F_{t,a})$ is increasing in a neighbourhood of $t_0$ as $\theta'(t_0)>0$.                                                                                                                     
Choose a compact subinterval $J'\subset J$ containing $0$. As $\theta'(t_0)=1$, we can choose a compact subinterval
$I'\subset I$ containing $t_0$ such that
\[m\eqdef \min_{t\in I'} \frac{\theta(t)-p/q}{t-t_0}>0.\]

For all $(t,x)\in I\times \R$, we have $\Phi(0,t,x) = p-q\theta(t)$. This implies that
\[\Phi(a,t,x) = p-q\theta(t) + a\Psi(a,t,x),\]
where $\Psi$ is an analytic map on a neighbourhood of $J\times I\times \R$. The function $\Psi$ is $1$-periodic with respect to $x$, which implies
that it is  bounded and reaches its bounds on $J'\times I'\times \R$. Take
\[M\eqdef \max_{(a,t,x)\in J'\times I'\times \R} \bigl|\Psi(a,t,x)\bigr|.\]

If $(t,a)\in I' \times J'\cap {\mathcal T}_{p/q}$, the function $\Phi$ vanishes on $\R$ and hence
\[|t-t_0|\leq \frac{q\theta(t)-p}{qm} \leq \frac{M|a|}{qm}.\]
Take $c_0\eqdef t_0$ and $M_1\eqdef M/(qm)$. If $(a,c_0+a\tau_1)\in I'\times J'\cap {\mathcal T}_{p/q}$, then $|\tau_1|\leq M_1$.

Now consider the map \[\Upsilon:(a,\tau_1,x)\mapsto \Phi(a,c_0+a\tau_1,x).\]
For all $(t,x)\in I'\times \R$ such that $t=c_0+a\tau_1$ and $\{(c_0+a\tau_1,a)\} \cap \mathcal T_{p/q} \ne \emptyset$ for $a$ in subinterval $J_0$ of $J'$ containing 0; 
we have $\Upsilon(0,\tau_1,x) = \Phi(0,t_0,x) = 0$. Which implies the existance of an analytic function $\Phi_1$ on a neighbourhood of
$\{0\}\times [-M_1,M_1]\times \R$ such that
\[\Phi(a,c_0+a\tau_1,x) = \Upsilon(a,\tau_1,x)= a\Phi_1(a,\tau_1,x).\]
Let $\phi_1:\R\to \R$ is defined by
\[\phi_1(x) \eqdef  \frac{\partial \Phi}{\partial a}(0,c_0,x).\]
Then,
\begin{align*}
\Phi_1(0,\tau_1,x) = \frac{\partial \Upsilon}{\partial a}(0,\tau_1,x) &=
\frac{\partial \Phi}{\partial a}(0,c_0,x) + \frac{\partial\Phi}{\partial t}(0,c_0,x)\cdot \tau_1 \\
&=\frac{\partial \Phi}{\partial a}(0,c_0,x) + \frac{\partial}{\partial t}(p-q\theta(t))|_{c_0}\cdot \tau_1\\
&= \phi_1(x)-q \tau_1.
\end{align*}

Suppose now that the statement is true for $n$; and assume that $n+1 \le q$. So if $(t,a)\in \mathcal T_{p/q}$,  then  $t = c_0 + ac_1 + \cdots + a^{n-1}c_{n-1}+ a^n \tau_n$ with $|\tau_n|\leq M_n$. Moreover
\begin{align*}
\Phi(a,c_0 + ac_1 + \cdots + a^{n-1}c_{n-1}+ a^n \tau_n,x) &= a^n\Phi_n(a,\tau_n,x)\\
&=a^n(\phi_n(x)-q\tau_n)+a^{n+1}\Psi_n(a,\tau_n,x);
\end{align*}
where $\phi_n$ is a real valued function defined on $\R$, also $\Phi_n$ and $\Psi_n$ are analytic with respect to $a, \tau_n$ and $x$ on a neighbourhood of $J_{n-1}\times [-M_n,M_n]\times \R$ with $0\in J_{n-1}\subset J'$.

From Proposition \ref{prop_periodicpart} we see that $\phi_n$ is $p/q$ periodic. This implies that
the all non zero terms in the Fourier series expansion of $\phi_n$ could only be those which are multiples of $q$. On the other hand we have a guided admissible family. Consequently by 5.6 $\phi_n$ is a trigonometric polynomial of degree $n$. By assumption $n<q$. Hence $\phi_n$ is a constant. Assume that $\phi_n(x)=qc_n$ for some constant $c_n$ for any $x$.

Take $\tau_n=c_n+a\tau_{n+1}$ so that $t=c_0+ac_1+\cdots+a^nc_n+a^{n+1}                                                                                                                                        
\tau_{n+1}$ and $(t,a)\in I' \times J' \cap \mathcal T_{p/q} \ne \emptyset$ for $a\in J_n\subset J'$. Then we have
\[\Phi(a,t,x)=a^{n+1}\bigl(-q\tau_{n+1}-\Psi_n(a,c_n+a\tau_{n+1},x)\bigr)=a^{n+1}
\Phi_{n+1}(a,\tau_{n+1},x)\]
for some function $\Phi_{n+1}$ analytic with respect to $a, \tau_{n+1}$ and $x$. Moreover $|\tau_{n+1}|\le M_{n+1}$ where
\[M_{n+1} \eqdef \frac{1}{q}\displaystyle \max_{(a,\tau_n,x)\in J'\times [-M_n,M_n]\times \R}\bigl| \Psi_n(a,\tau_n,x)\bigr|.\] Thus $\Phi_{n+1}$ and $\Psi_{n+1}$ are analytic on a neighbourhood of $J_n\times[-M_{n+1},M_{n+1}]\times \R$.

Define $\phi_{n+1}:\R \to \R$ such that $\phi_{n+1}(x)=\Psi_{n}(0,c_n,x)$. Which implies that
\[\Phi_{n+1}(0,\tau_{n+1},x)=\phi_{n+1}(x)-q\tau_{n+1}.\]
Hence we finish the proof by induction.
\end{proof}                                                                                                   

\begin{remark}
Recently Bonifant, Buff and Milnor have used the approach of Lemma \ref{lemma_induction} in their work \cite{buff-bonifant-milnor} for proving the existence of tongues in their family of cubic rational maps which are antipode preserving.  
\end{remark}

Now we are ready to discuss the main theorem of this section where we derive the order of contact of the boundaries of the rational tongues in admissible and guided families under some assumptions.
                                                            
\begin{theorem}\label{thm_orderofcontact1st}
Let $(F_{t,a}:\R\to \R)_{(t,a)\in I\times J}$ be an admissible family of maps guided by a holomorphic family
$(f_t:D_r\to \C)_{t\in I}$. We assume that $t_0\in I, \theta(t_0)=p/q$ with $\gcd(p,q)=1$, $\theta'(t_0)=1$ and
$f_{t_0}^{\circ q}(z) = z +C_{t_0}z^{q+1} + \mathcal{O}(z^{q+2})$ with $C_{t_0}\neq 0$. Then, there exist
\begin{itemize}
\item an interval $\hat J$ of $0$ with $\hat J\subset J$,

\item an interval $\hat I$ of $t_0$ with $\hat I\subset I$ and

\item two analytic maps $\gamma_{t_0}^\pm:\hat J\to \hat I$
\end{itemize}
such that
\begin{itemize}
\item $(t,a)\in {\mathcal T}_{p/q}\cap (\hat I\times \hat J)$ if and only if
$t$ is in between $\gamma_{t_0}^-(a)$ and $\gamma_{t_0}^+(a)$ and

\item $\displaystyle \gamma_{t_0}^+(a)-\gamma_{t_0}^-(a) \underset{a\to 0}\sim \frac{2|C_{t_0}|}{\pi q}a^{q}.$
\end{itemize}
\end{theorem}

\begin{remark}
For $a\geq 0$, $\gamma_{p/q}^-=\gamma_{p/q}^l$ and $\gamma_{p/q}^+=\gamma_{p/q}^r$. 
For $a<0$, $\gamma_{p/q}^-=\gamma_{p/q}^l$ and $\gamma_{p/q}^+=\gamma_{p/q}^r$ if $q$ is odd and  $\gamma_{p/q}^-=\gamma_{p/q}^r$ and $\gamma_{p/q}^+=\gamma_{p/q}^l$ if $q$ is even.
\end{remark}

\begin{proof}
We choose subintervals \[\hat I=I'\quad \text{and}\quad \hat J=  \bigcap_{k=0}^{q-1} J_k\] where $I'$ and $J_k$ are taken as Lemma \ref{lemma_induction}. For a fixed $a$, the set of $t$ such that
$(t,a) \in \mathcal T_{p/q}$ is an interval $[\gamma_{t_0}^-(a),\gamma_{t_0}^+(a)]$ for $a\ge0$. Thus $\gamma_{t_0}^\pm$ are defined for $a\ge 0$ first. We would try to show
that $\gamma_{t_0}^\pm$ are analytic near $a=0$. And their continuation for $a<0$ would be determined by that.
For $(t,a)\in \mathcal{T}_{p/q}\cap \hat I\times \hat J$, Lemma \ref{lemma_induction} implies that there is $M_q\ge 0$ and $t=c_0+ac_1+\cdots +a^{q-1}c_{q-1}+a^q\tau$ for $|\tau|\le M_q$ with
\[\Phi(a,c_0 + ac_1 + \cdots + a^{q-1}c_{q-1}+ a^q \tau,x) = a^q\Phi_q(a,\tau,x)=
a^q(\phi_q(x)-q\tau)+a^{q+1}\Psi_q(a,\tau,x)\]
where $\phi_q$ is a real valued analytic function defined on $\R$ and $\Psi_q$ is analytic with respect to $a, \tau$ and $x$. As $(t,a)\in \mathcal{T}_{p/q}$ there is $x'$ such that $\Phi(a,t,x')=0$. Which implies
\[\phi_q(x')-q\tau+a\Psi_q(a,\tau,x')=0.\]
Taking $a=0$ we see \[\phi_q(x')=q\tau.\]
By Proposition \ref{prop_adgddch}, $\phi_q$ is a trigonometric polynomial of degree $q$ and by Proposition \ref{prop_periodicpart} $\phi_q$ is 
$p/q$ periodic. This implies that there are $b_0,\beta \in \R$ and $b_q \in \C$ such that
\[\phi_q(x) =b_0+b_q e^{i2\pi qx}+\overline{b_q}e^{-i2\pi qx}=b_0+2|b_q|\sin(2\pi qx+\beta).\]

From previous calculations we know that
\[F_{t,a}^{\circ q}(x)=x+p-a^q(\phi_q(x)-q\tau)-a^{q+1}\Psi_q(a,\tau,x).\] The family
$(F_{t,a}:\R\to \R)_{(t,a)\in I\times J}$ is admissible and it is guided by the analytic family
$(f_t:D_r\to \C)_{t\in I}$. Which gives that $(F_{t,a}^{\circ q}:\R\to \R)_{(t,a)\in I\times J}$ is admissible and guided by
$(f_t^{\circ q}:D_{r'}\to \C)_{t\in I}$ for $r'<r$ such that $f_t^{\circ q}$ is defined on $D_{r'}$ for all $t\in I$. So we have an analytic family $(f_{t,a}^{\circ q}: A_{a,r'}\to \C^*)_{(t,a)\in I\times J_{r'}}$ where $J_{r'}=J\cap(-r',r')$ and $f_{t,a}^{\circ q}\circ \Pi_a(x)=\Pi_a \circ F_{t,a}^{\circ q}(x)$ for any $x$ and $(t,a)\in I\times J_{r'}'$. This means that
\[f_{t,a}^{\circ q}(ae^{i2\pi x})=\displaystyle
{ae^{i2\pi(x+p)}e^{i2\pi[-a^q(\phi_q(x)-q\tau)-a^{q+1}\Psi_q(a,\tau,x)]}}.\]
Replacing $ae^{i2 \pi  x}=z$ we see that
\[f_{t,a}^{\circ q}(z)=\displaystyle
ze^{2i\pi[-a^q(b_0-q\tau)-b_qz^q-a^{2q}\overline{b_q}z^{-q}+\mathcal{O}(z^{q+1})+\mathcal O(a)]}.\]
As $a\to 0$, we have
\[f_{t,a}^{\circ q}(z) \to ze^{i2\pi[-b_qz^q+\mathcal{O}(z^{q+1})]}= z(1-i2 \pi  b_qz^q+\mathcal O(z^{q+1})).\]
By assumption
\[f_{t,a}^{\circ q}(z) \underset{a \to 0}\to f_{t_0}^{\circ q}(z)=z(1+C_{t_0}z^q+\mathcal{O}(z^{q+1})).\]
This implies that $|b_q|=\displaystyle \frac{|C_{t_0}|}{2\pi}$. Thus $b_q\ne 0$ as $C_{t_0}\ne 0$.

%So \[\tau = \frac{1}{q}(b_0 \pm |b_q'|).\]
Define a function $G:(a,\tau,x)\mapsto \phi_q(x)-q\tau+a\Psi_q(a,\tau,x)$.
%For $(t,a)$ in the boundary of $\mathcal{T}_{p/q}$
We obtain the following equations 
\begin{align}
G(0,\tau,x) =\phi_q(x)-q\tau& =b_0+2|b_q|\sin(2\pi qx+\beta)-q\tau=0, \\
\frac{\partial}{\partial x}G(0,\tau,x)= \frac{\partial \phi_q(x)}{\partial x}& =  4\pi q|b_q|\cos(2\pi qx+\beta)=0.
\end{align}
There are two sets of solutions to this above system 
\[\tau^\pm(0) = \displaystyle\frac{1}{q}(b_0 \pm 2|b_q|)\quad \text{and}\quad 2\pi qx^+(0)+\beta=\pi/2, 2\pi qx^-(0)+\beta=3\pi/2.\]
Using implicit Function theorem we would show that $\tau^\pm(a)$ and $x^\pm(a)$ can be expressed as analytic functions of $a$ starting with these two solutions near $a=0$. Note that
\begin{align*}
& \frac{\partial G}{\partial \tau}(0,\tau^\pm(0),x^\pm(0))=-q,& \frac{\partial G}{\partial x}(0,\tau^\pm(0),x^\pm(0))=0, \\
& \frac{\partial}{\partial \tau}\frac{\partial G}{\partial x}(0,\tau^\pm(0),x^\pm(0))=0, &\frac{\partial^2 G}{\partial x^2}
(0,\tau^\pm(0),x^\pm(0))=\pm8\pi^2q^2|b_q|\ne0.
\end{align*}
Therefore the concerned matrices
${\begin{pmatrix}
-q & 0 \\
0 & \pm8\pi^2q^2|b_q|
\end{pmatrix}}$
are invertible corresponding to the two sets of solutions $\{\tau^+(0),x^+(0)\}$ and $\{\tau^-(0),x^-(0)\}$. Starting from these two sets of solutions, $\tau^\pm(a)$ can be expressed analytically as a function of $a$ near $a=0$.
This implies that if $(t,a)$ is in the boundary of $\mathcal T_{p/q}$ and $(t,a)\in \hat I\times \hat J$ then
\[t=c_0+ ac_1+\cdots+a^{q-1}c_{q-1}+a^q\tau^\pm(a).\]
And this means that $\gamma_{t_0}^\pm(a)$ are analytic in $a$ near $0$. We have defined $\gamma_{t_0}^\pm(a)$ for $a\ge 0$. Now when we know that $\gamma_{t_0}^\pm(a)$ are parts of analytic curves near $a=0$, following these curves we can define $\gamma_{t_0}^\pm(a)$ accordingly for $a<0$. If $q$ is even and $a<0$ we define $\gamma_{t_0}^\pm(a)$ as the maximum and minimum values of $t$ such that
$(t,a) \in \mathcal T_{p/q}$. When $q$ is odd and $a<0$ we define $\gamma_{t_0}^\pm(a)$ as the minimum and maximum values of $t$ such that $(t,a) \in \mathcal T_{p/q}$. Now it remains to prove the estimate of the difference
$\gamma_{t_0}^+(a)-\gamma_{t_0}^-(a)$ near $a=0$.

We have seen that $\tau^\pm(0)=\ds\frac{1}{q}(b_0\pm2|b_q|)$.
Hence,
\begin{align*}
\gamma_{t_0}^+(a)-\gamma_{t_0}^-(a)&=(\tau^+(0)-\tau^-(0))a^q+\mathcal
O(a^{q+1})\\
& \underset{a \to 0}\sim \frac{2|C_{t_0}|}{\pi q}a^q.
\end{align*}
\end{proof}

\subsection{Application to the Standard family}

Our next goal is to prove that in the case of the standard family $(S_{t,a}:x \mapsto x+t+a\sin (2\pi x))$ the order of contact of the boundaries of the rational tongue $\mathcal T_{p/q}$ is exactly $q$. For proving this we need the following proposition.

\begin{proposition}\label{prop_cyclestdfmly}
Let $I$ and $J$ be open intervals containing $p/q$ and $0$ such that $I=\R$ and $J=(-1/2\pi,1/2\pi)$. Then the standard family $(S_{t,a}:x \mapsto x+t+a\sin (2\pi x))_{(t,a)\in I\times J}$ is admissible and guided by $(s_t :D_r\to \C)_{t \in I }$ such that $s_t (z)=e^{i2 \pi  t}ze^{\pi z}$. Moreover there is a constant $C_{p/q}\ne 0$ such that $s_{p/q} ^{\circ q}(z)=z+C_{p/q}z^{q+1}+\mathcal O(z^{q+2})$.
\end{proposition}

\begin{proof}
We have already seen that the standard family $(S_{t,a} :x \mapsto x+t+a\sin (2\pi x))_{(t,a) \in I\times J}$ is admissible and guided by $(s_t:D_r\to \C)_{t\in I}$ such that $s_t:z\mapsto e^{i2\pi}ze^{\pi z}$.

The map $s_{p/q} $ has one critical point and one asymptotic value at $0$, which is a parabolic fixed point. Therefore there is only one cycle of petals. This implies that there is a non zero constant $C_{p/q}$ such that $s_{p/q} ^{\circ q}(z)=z+C_{p/q}z^{q+1}+\mathcal O(z^{q+2})$.
\end{proof}

An immediate consequence of Proposition \ref{prop_cyclestdfmly} with Theorem \ref{thm_orderofcontact1st} is the following Theorem on the order of contact of the boundaries of the rational tongues in the standard family. 

\begin{theorem}\label{theo_orderofcontactstd}
Let $I$ and $J$ be open intervals containing $p/q$ and $0$. The standard family $(S_{t,a} :x \mapsto x+t+a\sin (2\pi x))_{(t,a) \in I\times J}$ is admissible and guided by 
$(s_t :D_r\to \C)_{t \in I }$ such that $s_t (z)=e^{i2 \pi  t}ze^{\pi z}$. For $t=p/q$
there is a constant $C_{p/q}\ne 0$ such that $s_{p/q} ^{\circ q}(z)=z+C_{p/q}z^{q+1}+\mathcal O(z^{q+2})$.
Moreover there exists
\begin{itemize}
\item an interval $\hat J$ of $0$ with $\hat J\subset J$ and

\item an interval $\hat I$ of $p/q$ with 

\item two analytic maps $\gamma_{p/q}^\pm:\hat J\to \hat I$
\end{itemize}
such that
\begin{itemize}
\item $(t,a) \in {\mathcal T}_{p/q}\cap (\hat I\times \hat J)$ if and only if
$t$ is in between $\gamma_{p/q}^-(a)$ and $\gamma_{p/q}^+(a)$ and

\item $\displaystyle \gamma_{p/q}^+(a)-\gamma_{p/q}^-(a) \underset{a\to 0}\sim \frac{2|C_{p/q}|}{\pi q}a^{q}.$
\end{itemize}
\end{theorem}

Now we have proved that the order of contact of the boundaries of $\cal T_{p/q}$ is exactly $q$ in the standard family. The dependance of $|C_{p/q}|$ on $p/q$ has been studied by Ch\'eritat in his PhD thesis \cite{arnaud} (this is related to the asymptotic size and to the conformal radius of Siegel disks). This behaviour and its connections to the
Brjuno function have then been more extensively studied by Buff and
Ch\'eritat. Using the results of Buff and Ch\'eritat, our previous result
partially answers questions raised by Broer, S\'imo and Tatjer \cite{broeretal}

In the course of proving Theorem \ref{theo_orderofcontactstd} we also showed that the boundary curves $\gamma_{p/q}^\pm$ are analytic functions of the variable $a$ near $a=0$. We would apply this fact in proving that the boundaries $\gamma_{p/q}^\pm$ are analytic in the standard family. 

\begin{theorem}\label{theo_bdanalytic}
The boundary curves of $\cal T_{p/q}$ are analytic functions in the standard family within the parameter space $\cal P_S$.
\end{theorem}

\begin{proof}
By Theorem \ref{theo_orderofcontactstd} the boundary curves $\gamma_{p/q}^\pm$ are analytic in $a$ in a neighbourhood of 0. Precisely there are intervals $\hat J\ni 0$ and $\hat I$ such that $\gamma_{p/q}^\pm:\hat J \to \hat I$ are analytic. From the Theorem \ref{theo_bdrystdanalytic1st} we know that the boundaries $\gamma_{p/q}^\pm$ are analytic functions of $a$ for $a\ne 0$. Considering these two results together it is proved that the boundary curves of $\cal T_{p/q}$ in the standard family are analytic functions of $a$.
\end{proof}

\subsection{Application to the Blaschke family}

We can study the Blaschke fraction family like the standard family. This family behaves pretty much like the standard family when we consider the order of contact of the boundaries of $\cal T_{p/q}$ or their analyticity. Let us prove that the order of contact of the boundaries of $\cal T_{p/q}$ in the Blaschke family is exactly $q$. 

\begin{proposition}
Let $I=\R$ and $J=(-1/\sqrt2,1/\sqrt2)$. The Blaschke family $(B_{t,a}:x\mapsto x+t+ 2\arctan \ds\frac{a\sin(2\pi x)}{1-a\cos(2\pi x)})_{(t,a)\in I\times J}$ is admissible and guided by the quadratic family $(b_t:w\mapsto e^{i2\pi t}w(1-w))_{t\in I}$. The map $b_{p/q}$ is such that $b_{p/q}^{\circ q}(w)=w+B_{p/q}w^{q+1}+\cal O(w^{q+2})$ where $B_{p/q}\ne 0$ is a constant.
\end{proposition}

\begin{proof}
The first part of the proposition follows from Example \ref{example_Bla}. And the map $b_{p/q}$ has a parabolic fixed point at 0 with multiplier $e^{i2\pi p/q}$. We also note that $b_{p/q}$ has only one critical point being a quadratic polynomial. Also this map does not have any finite asymptotic value. Thus $b_{p/q}$ has only a single cycle of petals. Therefore there exists a constant $B_{p/q}\ne 0$ such that 
\[ b_{p/q}^{\circ q}(w)=w+B_{p/q}w^{q+1}+\cal O(w^{q+2}).\]
\end{proof}

Arguing exactly like Theorem \ref{theo_orderofcontactstd} we obtain the order of contact in this case.

\begin{corollary}
The order of contact of the boundaries of $\cal T_{p/q}$ in the Blachke family is exactly $q$.
\end{corollary}

Exactly same like Theorem \ref{theo_bdanalytic} one can prove that the boundary curves of the tongue $\cal T_{p/q}$ are always analytic in this family. 

\begin{corollary}
The boundary curves of the tongue $\cal T_{p/q}$ are analytic in the Blaschke family within the parameter space.
\end{corollary}

{\bf Acknowledgments:} This paper is a part of the author's doctoral thesis, which was funded
mainly by the EU Research Training Network on Conformal Structures and Dynamics
(CODY), Marie-Curie Research Training Networks and partially by CNRS and ANR grant
ANR-08-JCJC-0002. The author would like to thank his advisor Xavier Buff for posing
the problem and his guidance.  

\bibliographystyle{amsalpha}

\end{document}